\documentclass[12pt]{amsart}
\newtheorem{thm}{Theorem}[section]
\newtheorem{cor}[thm]{Corollary}
\newtheorem{lem}[thm]{Lemma}

\theoremstyle{definition}
\newtheorem{defn}[thm]{Definition}
\theoremstyle{remark}
\newtheorem{rem}[thm]{Remark}
\numberwithin{equation}{section}

\def\d{\Delta}

\def\L2{L^{2}}

\def\M{\mathcal{M}}
\def\E{\mathcal{E}}
\def\B{\mathcal{B}}
\def\D{\mathcal{D}}

\def\L{\mathcal{L}}

\def\R{\mathbb{R}}
\def\C{\mathbb{C}}
\def\Z{\mathbb{Z}}

\def\m1{^{-1}}

\def\C{\mathcal{C}}
\def\H{\mathcal{H}}
\def\F{\mathcal{F}}
\def\J{\mathcal{J}}
\def\K{\mathcal{K}}
\def\o{\omega}
\def\O{\Omega}
\def\P{\mathcal{P}}

\begin{document}

\title[]{The emergence of \\Noncommutative Potential Theory.}%
\author{Fabio Cipriani}%
\address{Dipartimento di Matematica, Politecnico di Milano, piazza Leonardo da Vinci 32, 20133 Milano, Italy.} \email{fabio.cipriani@polimi.it}
\footnote{This work has been supported by Laboratoire Ypatia des Sciences Mathématiques C.N.R.S. France - Laboratorio Ypatia delle Scienze Matematiche I.N.D.A.M. Italy (LYSM).}
\date{September 19, 2020}
% ----------------------------------------------------------------
\maketitle
\begin{abstract}
We review origins and developments of Noncommutative Potential Theory as underpinned by the notion of energy form. Recent and new applications are shown to approximation properties of von Neumann algebras.
\end{abstract}

\tableofcontents
% ----------------------------------------------------------------------------------------------------------------
\section{Introduction.}
Our intent here is to trace some of the main steps of Noncommutative Potential theory, starting from the seminal works by Sergio Albeverio and Raphael Hoegh-Khron [AHK1,2]. The point of view adopted in treating Potential Theory it is essentially the one of Dirichlet forms, i.e. the point of view of Energy. The justification for this is that, not only the motivating situations to develop a potential theory on operator algebras came from Mathematical Physics but also that the concept of Energy seems to have a unifying character with respect to the different aspects of the subject.\\
The present exposition is thought to be addressed to researcher not necessarily familiar with the tools of operator algebras and, in this respect, we privileged the illustration of examples and applications instead to provide the details of the proofs.\\
In this presentation several aspects of the theory has been necessarily sacrificed and for them we refer to other presentations [Cip4], [Cip5]. In particular, the construction of Fredholm modules and Dirac operators from Dirichlet forms and the realization of Dirichlet spaces as istances of A. Connes' Noncommutative Geometry [Co5] can be found in [Cip4],[CS3] and [S4,7] while the study of energy states, potentials and multipliers of noncommutative Dirichlet spaces has been initiated in [CS4]. The details of the theory on KMS symmetric Markovian semigroups on C$^*$-algebras can be found in [Cip5].\\
The recent developments of the theory of noncommutative Dirichlet forms show a not rare situation in Mathematics in which a theory born to solve specific problems, as time goes by, applies to, apparently far away, others. In this respect we review in Section 7 the recent close relationships among spectral characteristics of noncommutative Dirichlet forms and approximation properties of von Neumann algebras such as Haagerup Property (H), amenability and Property (T). In particular a new characterization of the Murray-von Neumann Property ($\Gamma$) is proved in terms of absence of a Poincar\'e inequalitiy for elementary Dirichlet forms.\\

\section{Commutative Potential theory}
\subsection{Classical Potential Theory}
Classical Potential Theory concerns properties of the Dirichlet integral
\[
\D:L^2 (\mathbb{R}^d,m)\to [0,+\infty]\qquad \D[u]:=\int_{\mathbb{R}^d} |\nabla u|^2\,dm
\]
as a lower semicontinuous quadratic form on the Hilbert space $L^2 (\mathbb{R}^d,m)$, which is finite on the Sobolev space $H^1 (\mathbb{R}^d)$. The associated positive, self-adjoint operator is the Laplace operator
\[
\Delta=-\sum_{k=1}^d\partial_k^2\qquad\D[u]=\|{\sqrt \Delta}u\|^2_2
\]
which generates the heat semigroup $e^{-t\Delta}$ whose Gaussian kernel
\[
e^{-t\Delta}(x,y)=(4\pi t)^{-d/2}e^{-\frac{|x-y|^2}{4t}}
\]
is the fundamental solution of the heat equation $\partial_t u +\Delta u=0$. The Brownian motion $(\Omega, P_x, B_t)$ is the stochastic processes associated to the semigroup by the relation
\[
(e^{-t\Delta}u)(x)=\mathbb{E}_x(u\circ B_t)
\]
which is also directly connected to the Dirichlet integral by the identity
\[
\D[u]=\lim_{t\to 0^+}\frac{\mathbb{E}_m(|u\circ B_t-u\circ B_0|^2)}{2t}\, .
\]
The polar sets, i.e. those sets which are avoided by the Brownian motion, can be characterized as those which have vanishing electrostatic capacity, defined in terms of the Dirichlet integral itself as
\[
{\rm Cap}(A):=\inf\{\D[u]+\|u\|^2_2:u\in H^1(\R^n),\,\,\, 1_A\le u\}
\]
for any open set $A\subseteq\R^n$ and then as
\[
{\rm Cap}(B):=\inf\{{\rm Cap}(A): B\subseteq A,\,\,\, A\,\,\,{\rm open}\}
\]
for any other measurable set $B\subseteq\R^n$.
The heat semigroup is Markovian on $L^2(\R^n,m)$ in the sense that it is strongly continuous, contractive, positivity preserving and satisfies $e^{-t\Delta}u\le1$ whenever $u$ is a real function such that $u\le 1$. By these properties it can be extended to a contractive and positivity preserving semigroup on any $L^p(\R^n,m)$ for $p\in [1,+\infty]$ which is strongly continuous for $p\in [1,+\infty)$ and weakly$^*$-continuous for $p=+\infty$. The Markovianity of the heat semigroup is equivalent to the following property, also called Markovianity, of the Dirichlet integral
\[
\D[u\wedge 1]\le \D[u]\qquad u={\bar u}\in L^2(\R^n,m)
\]
which can be easily checked using differential calculus and the definition of the Dirichlet integral. All others above properties can be proved by the explicit knowledge of the Green kernel of heat semigroup which, for $d\ge 3$ at least, equals
\[
\Delta^{-1} u(x)=\int_{\mathbb{R}^d}G(x,y)u(y)\, m(dy)\qquad G(x,y)=c_d\cdot |x-y|^{2-d}\, .
\]

\subsection{Beurling-Deny Potential Theory} ([BD2], [Den], [CF], [FOT], [LJ], [Sil1], [Sil2]).
A turning point in the development of potential theory was represented by two seminal papers by A. Beurling and J. Deny [BD1,2]. They developed a {\it kernel free potential theory} based on the notion of {\it energy} on general locally compact measured spaces $(X,m)$. The whole theory relies on the notion of {\it regular Dirichlet form} which is required to be a lower semicontinuous quadratic functional on $L^2(X,m)$ satisfying
\begin{itemize}
\item {\it Markovianity}
\[
\E:L^2(X,m)\to [0,+\infty]\qquad \E[u\wedge 1]\le\E[u]
\]
\item {\it regularity}: $\F\cap C_0(X)$ is a form core uniformly dense in $C_0(X)$
\end{itemize}
and where  the form domain $\F:=\{u\in L^2(X,m): \E[u]<+\infty\}$ is assumed to be $L^2$-dense. The lower semicontinuity of $\E$ on $L^2(X,m)$, being equivalent to the closedness of the densely defined quadratic form $(\E,\F)$ on $L^2(X,m)$, implies the existence of a  nonnegative, self-adjoint operator $(L,D(L))$ which generates a Markovian semigroup $e^{-tL}$ on $L^2(X,m)$.
\subsubsection{Beurling Deny decomposition}
One of the first fundamental results in the Beurling-Deny analysis concerns the structure of a general regular Dirichlet form: these can be uniquely realized as a sum of three Markovian forms (each of which not necessarily closed)
\[
\E=\E^d+\E^j+\E^k
\]
where the {\it jumping part} has the form
\[
\E^j[u]=\int_{X\times X\setminus\Delta_X} |u(x)-u(y)|^2\, J(dx,dy)
\]
for a positive measure $J$ supported off the diagonal $\Delta_X$ of $X\times X$, the {\it killing part} appears as
\[
\E^k[u]=\int_X |u(x)|^2\, k(dx)
\]
for some positive measure $k$ on $X$ and the {\it diffusion part is strongly local} in the sense that
\[
\E^d[u+v]=\E^d[u]+\E^d[v]
\]
whenever $u$ is constant in a neighborhood of the support of $v$.
\vskip0.2truecm
Two turning point in the development of Potential Theory took place on the probabilistic side when M. Fukushima associated a Hunt stochastic process $(\O,P_x, X_t)$ to a regular Dirichlet form in such a way that
\[
\E[u]=\lim_{t\to 0^+}\frac{\mathbb{E}(|u\circ X_t-u\circ X_0|^2)}{2t}\qquad u\in\F
\]
and when M. Silverstein introduced the notion of {\it extended Dirichlet space}, especially for the connections with the boundary theory and the random time change of symmetric Hunt processes.
\vskip0.2truecm\noindent
The process is a stochastic dynamical system which represents the semigroup through
\[
(e^{-tL}u)(x)=\mathbb{E}_x(u\circ X_t)\, .
\]
A basic tool in the development of the Beurling-Deny theory of a regular Dirichlet form $(\E,\F)$ is the capacity one associates to it exactly in the same way we have seen above in the case of Dirichlet integrals. A key point to construct the associated stochastic processes is the fact that the regularity property of the Dirichlet form allows to prove that the capacity associated to it is in fact a {\it Choquet capacity} which implies that Borel sets are capacitable.\\
From the point of view of the process, the three different summands of the Beurling-Deny decomposition have a nice and useful probabilistic interpretation: the measure $J$ counts the jumps of the process, the measure $k$ specifies the rate at which the process is killed inside $X$ and a Dirichlet form is strongly local if and only if the associated process is  a diffusion, i.e. it has continuous sample paths. In Section 5.2 we will show an independent, algebraic way to prove the above decomposition of Dirichlet forms.

%\subsection{$\infty$-dimensional Markov Processes and the role of QFT for bosons}
%Together with R. Hoegh-Krohn, Sergio Albeverio extended Dirichlet form theory from the realm of locally compact topological spaces to the one of infinite dimensional Banach spaces. One of the main difficulties to overcome, especially for the benefits of the probabilistic side, is to suitably generalize the regularity assumption. Here we want just to emphasize the role played in these developments by the needs of Euclidean QFT for boson systems or, which is the same (see Guerra-Rosen-Simon), by Statistical Mechanics. See Simon book.

\section{Operator algebras}

\subsection{C$^*$-algebras as noncommutative topology} ([Arv], [Dix1], [Ped], [T2]).\\
 A C$^*$-algebra is $A$ is a Banach $*$-algebra in which norm and involution conspire as follows
\[
\|a^* a\|=\|a\|^2\qquad a\in A\, .
\]
This notion generalizes topology in an algebraic form in the sense that, by a theorem of I.M. Gelfand, a {\it commutative} C$^*$-algebra $A$ is isomorphic to the algebra $C_0(X)$ of continuous functions vanishing at infinity on a locally compact Hausdorff space $X$, called the {\it spectrum} of $A$. In $C_0(X)$ the product of functions is defined pointwise, the involution is given by pointwise complex conjugation and the norm is the uniform one.\\
The simplest example of a {\it noncommutative} C$^*$-algebra is the full matrix algebra $M_n(\mathbb{C})$ where the product is the usual rows-by-columns, the involution of a matrix $A$ is defined as its matrix adjoint $A^*$ and the norm $\|A\|$ is given by the operator norm (the largest singular value of $A$, i.e. the square root of the largest eigenvalue of $A^*A$).\\
{\it Finite dimensional} C$^*$-algebras are isomorphic to finite direct sums of full matrix algebras.
The simplest examples of noncommutative, {\it infinite dimensional} C$^*$-algebras are those of the algebra $B(h)$ of all bounded operators and its  subalgebra of all compact operators $\mathcal{K}(h)$ on an infinite dimensional Hilbert space $h$, the norm being the operator one.\\
%while a less trivial example is that of the Clifford algebra ${\rm Cl}(h)$ that we will encounter below.\\
%Various operation in the category of C$^*$-algebras such as finite direct sum, direct integral, tensor product, inductive and projective limits allow to construct new algebras from older ones and to analyze complicate algebras in terms of simpler ones. Other construction of C$^*$-algebras come naturally from the structure of pre-existing objects such as groups, groupoids, Riemannian metrics, foliations and group actions arising in various context such as Harmonic Analysis, Geometry and Dynamical Systems.\\
A morphism $\alpha:A\to B$ between C$^*$-algebras $A,B$ is a norm continuous $*$-algebras morphism. A first example of the deep interplay that the algebraic and the analytic structures on a C$^*$-algebra give rise, is the fact that $^*$-algebra morphisms are automatically norm continuous. Morphisms between commutative C$^*$-algebras $C_0(X)$ and $C_0(Y)$ correspond to homeomorphisms $\phi:Y\to X$ by $\alpha(f)=f\circ\phi$.\\
A morphism of type $\pi:A\to B(h)$ is called a representation of $A$ on the Hilbert space $h$. It is called faithful if it is an injective map and in this case $A$ can be identified with the C$^*$-subalgebra $\pi(A)\subseteq B(h)$. Any C$^*$-algebra admits a faithful representation.\\
A C$^*$-algebras is, in particular, an ordered vector space where the closed cone is given by
\[
A_+:=\{a^* a\in A: a\in A\}.
\]
When $A$ is represented as a subalgebra of some $B(h)$, the positive elements of $A$ are positive, self-adjoint operators on $h$. In $C_0(X)$, the positive elements are just the nonnegative functions.
\subsection{von Neumann Algebras as noncommutative measure theory} ([Dix2], [MvN], [Ped]).
A von Neumann algebra $M$ is a C$^*$-algebra which admits a predual $M_*$ as a Banach space in the sense that $(M_*)^*=M$.\\
Any commutative, $\sigma$-finite\footnote{A von Neumann algebra is $\sigma$-finite if all collections of mutually disjoint orthogonal projections have at most a countable cardinality. von Neumann algebras acting on separable Hilbert spaces are $\sigma$-finite (the converse being in general not true).} von Neumann algebra is isomorphic to the algebra of (classes of) essentially bounded measurable functions $L^\infty (X,m)$ on a measured standard space $(X,m)$ with $L^1(X,m)$ as predual space. This commutative situation forces to regard the theory of von Neumann algebras as a noncommutative generalization of Lebesgue measure theory. Even if this is a fruitful point of view, other natural constructions suggest to look at the theory as a generalization of Euclidean Geometry and as a generalization of Harmonic Analysis.\\
The simplest example of a noncommutative von Neumann algebra is that of the space $B(h)$ of all bounded operators acting on a Hilbert space $h$ having dimension greater than one. The predual of $B(h)$ is given by the Banach space $L^1(h)$ of trace-class operators on $h$ and the duality is given by
\[
\langle A,B\rangle:={\rm Tr}(AB)\qquad A\in B(h),\quad B\in L^1(h)\, .
\]
All C$^*$-algebras are isomorphic to norm-closed subalgebras of some $B(h)$ and all von Neumann algebras are isomorphic to subalgebras of some $B(h)$, closed in its weak$^*$-topology. A first fundamental results of J. von Neumann asserts that for any subset $S\subseteq B(h)$, its {\it commutant}
\[
S':=\{a\in B(h): ab=ab,\,\,\,{\rm for\,\,all}\,\,\, b\in B(h)\}
\]
is a von Neumann algebra. A second fundamental result of J. von Neumann asserts that an involutive subalgebra $M\subseteq B(h)$ is a von Neumann algebras iff it is weakly$^*$-closed and iff it coincides with its double commutant $M=M'':=(M')'$. A key aspect is that for an involutive subalgebra $M\subseteq B(h)$ its weak$^*$-closure coincides with its double commutant $(M')'$.\\
The center of a von Neumann algebra is defined as
\[
{\rm Center}(M):=\{x\in M:xy=yx,\,\, y\in M\}.
\]
and $M$ is called a {\it factor} if its center reduces to the one dimensional algebra $\mathbb{C}\cdot 1_M$ of scalar multiples of the unit of $M$.

\subsection{Weights, traces, states and the GNS representation} ([Dix1], [Ped]).
A {\it positive functional} on a $C^*$-algebra $A$ is a linear map $\tau:A\to \mathbb{C}$ such that
\[
\tau(a)\ge 0\qquad a\in A_+.
\]
These are automatically bounded and are called {\it states} when having norm one. In case $A$ has a unit, a positive functional is a state as soon as $\tau(1_A)=1$ as it follows from $0\le a\le \|a\|_A\cdot1_A$. Positive functionals are noncommutative analog of finite, positive Borel measures on locally compact spaces: in fact, by the Riesz Representation Theorem, a positive functional on a commutative C$^*$-algebra $C_0(X)$ corresponds, via Lebesgue integration, to a finite, positive Borel measure $m$ on $X$
\[
\tau(a)=\int_X a\,dm\qquad a\in C_0(X),
\]
which is a probability if and only if $\tau$ is a state. To accommodate the analog of possibly unbounded positive Borel measures, one has to consider {\it weights} on $A$  defined as functions $\tau:A_+\to [0,+\infty]$ which are {\it homogeneous} and {\it additive} in the sense
\[
\tau(\lambda a)=\lambda\tau(a)\, ,\qquad \tau (a+b)=\tau(a)+\tau(b) \qquad a,b\in A_+,\quad \lambda\ge 0\, .
\]
If a weight is everywhere finite, then it can be extended to a positive linear functional on $A$. A weight is called a {\it trace} if it is invariant under inner automorphisms  in the sense
\[
\tau(uau^*)=\tau(a)\, ,\qquad a\in A_+
\]
for all unitaries $u\in \widetilde A=A\oplus\mathbb{C}$ (recall that $A$ is a two-sided ideal in $\widetilde A$). This is equivalent to require that $\tau$ is {\it central} in sense that
\[
\tau (a^*a)=\tau(aa^*)\qquad a\in A.
\]
If $\tau$ is finite this reduces to
\[
\tau(ab)=\tau(ba)\qquad a,b\in A.
\]
A weight is {\it faithful} if it vanishes $\tau(a)=0$ on $a\in A_+$ only when $a=0$. In the commutative case, faithful weights correspond to fully supported positive Borel measures. \\
A weight is {\it densely defined} if the ideal $A^\tau:=\{a\in A_+:\tau(a^*a)<+\infty\}$ is dense in $A$. If a trace is lower-semicontinuous, than it is {\it semifinite} in the sense that
\[
\tau(a)=\sup\{\tau(b)\in A_+:b\le a\}\qquad b\in A_+.
\]
On a von Neumann algebra, a weight is {\it normal} if
\[
\tau(\sup_{i\in I} a_i)=\sup_{i\in I}\tau(a_i)
\]
for any net $\{a_i:i\in I\}\subset A_+$ admitting a least upper bound in $A_+$. The predual Banach space $M_*$ of a von Neumann algebra $M$ can be shown to the space of all normal continuous functionals on $M$. A von Neumann algebra is said to be {\it finite} (resp. {\it semi-finite}) if, for every non-zero $a\in A_+$, there exists a finite (resp. semi-finite) normal trace $\tau$ such that $\tau(a)>0$ and it is said {\it properly infinite}
(resp. {\it purely infinite}) if the only finite (resp. semi-finite) normal trace on $A$ is zero. On a semi-finite von Neumann algebra, there exists a semi-finite faithful normal trace.
\vskip0.1truecm\noindent
In this exposition we will be essentially concerned with semi-finite von Neumann algebras and in particular with those which are  {\it $\sigma$-finite} in the sense that they admit a faithful, normal state. If $h$ is a separable Hilbert space, then any von Neumann algebra $A\subseteq B(h)$ is $\sigma$-finite. In fact, for any Hilbert base $\{e_k\in h: k\in\mathbb{N}\}$, a faithful, normal state is provided by
\[
\tau(x):=\sum_{k\in\mathbb{N}}(e_k|xe_k)_h\qquad x\in A.
\]
In a way similar to the one by which a probability measure $m$ on $X$ give rise to the Hilbert space $L^2(X,m)$ and to the representation of continuous functions in $C_0(X)$ as multiplication operators on it, a densely defined weight $\tau$ on a C$^*$-algebra $A$ give rise to a Hilbert space $L^2(A,\tau)$ on which the elements $a\in A$ act as bounded operators. This is called the Gelfand-Neimark-Segal or GNS-representation of $A$ associated to $\tau$.\\ In fact, the sesquilinear form $x,y\mapsto \tau(x^*y)$ on the vector space $A$, satisfies the {\it Cauchy-Schwarz inequality}
\[
|\tau (x^*y)|^2\le\tau(x^* x)\tau (y^* y)\qquad x,y\in A^\tau
\]
and a Hilbert space $L^2(A,\phi)$ can be constructed from the inner product space $A^\tau$ by separation and completion. Since $A^\tau$ is an ideal of $A$, the left regular action $b\mapsto ab$ of $A$ onto itself give rise to an action of $A$ onto $A^\tau$ and then to a representation of $A$ on the GNS Hilbert space. If $\tau$ is faithful, the identity map of $A$ give rise to an injective, bounded map $A\to L^2(A,\tau)$ and if $A$ is unital, the vector $\xi_\phi\in L^2(A,\tau)$ image of the identity $1_A\in A$, allows to represent the state $\phi$ by $\tau (x)=(\xi_\tau |x\xi_\tau)_2$. This vector, uniquely determined by this property, is {\it cyclic} in the sense that $\overline{A\xi_\tau}=L^2(A,\tau)$ and {\it separating} in the sense that if $a\in A$ and $a\xi_\tau=0$ then $a=0$.\\
The von Neumann algebra $L^\infty(A,\tau):=(\pi_{\rm GNS}(A))''   \subseteq B(L^2(A,\tau))$ obtained by w$^*$-completion, is called the {\it von Neumann algebra generated by $\tau$ on $A$}. The GNS-representation can then be extended to a normal representation of $L^\infty(A,\tau)$. As notations are aimed to suggest, this is a generalization of the usual construction of Lebesgue-Riesz measure theory.
\vskip0.2truecm\noindent
In the case of the trace functional ${\rm Tr}$ on $\mathcal{K}(h)$, the associated GNS space is $L^2(\mathcal{K}(h),{\rm Tr})=L^2(h)$ the space of Hilbert-Schmidt operators on which compact operators in $\mathcal{K}(h)$ act by left composition.\\
The Hilbert space of the GNS representation of a faithful trace is naturally endowed with a closed convex cone $L^2_+(A,\tau)$, which provides an order structure on $L^2(A,\tau)$. It is defined as the closure of $A^\tau$. In the commutative case $A=C_0(X)$, this is just the cone of square integrable, positive functions. $L^2(X,m)$. The construction of a suitable closed, convex cone from a faithful state on a C$^*$-algebra or from a faithful normal state on a von Neumann algebra will be done later on.\\
We conclude this section mentioning that a noncommutative integration theory for traces on C$^*$-algebras has been developed in [Ne], [Se] giving rise to an interpolation scale of spaces $L^p(A,\tau)$ between the von Neumann algebra $L^\infty(A,\tau)$ and its predual $L^1(A,\tau)$. The elements of this spaces can realized as closed operators on $L^2(A,\tau)$.

\subsection{Morphisms of operator algebras} ([Ped]).The most obvious notion of morphism to form a category of C$^*$-algebras is certainly that of continuous morphisms of involutive algebras. However, this category risks to have a poor amount of morphisms. For example, if $\alpha:A\to B$ is a morphism and $B$ is commutative then
$\alpha (ab-ba)=\alpha (a)\alpha(b)-\alpha (b)\alpha(a)=0$ so that if the algebra generated by commutators $[a,b]:=ab-ba$ is dense in $A$ then $\alpha =0$. This is the case for example of $A=\mathcal{K}(h)$  or more generally for the so called {\it stable} C$^*$-algebras.\\
We illustrate now a much more well behaved notion of morphism between C$^*$ and von Neumann algebras, i.e. {\it completely positive map} (an even more general and fundamental notion of morphism is that of {\it Connes correspondence}, which we will meet later on in this lectures). This notion is of probabilistic nature in the sense that, among commutative von Neumann algebras, completely positive maps are just the transformations associated to {\it positive kernels}. Notice, {\it en passant}, that another basic tool in operator algebra theory which is of clearly probabilistic nature is the notion of {\it conditional expectation}. See discussion in [Co5 Chapter 5 Appendix B].
\vskip0,5truecm\noindent
Beside to any C$^*$-algebra $A$ we may consider its {\it matrix ampliations} $A\otimes \mathbb{M}_n(\mathbb{C})$, $n\ge 1$. A linear map $T:A\to B$ is said to be {\it completely positive}, or CP map, if its ampliations
\[
T\otimes I_n:A\otimes \mathbb{M}_n(\mathbb{C})\to B\otimes \mathbb{M}_n(\mathbb{C})
\]
are positive for any $n\ge 1$. $*$-algebra morphisms (such as representations) are completely positive maps. If $A$ or $B$ is commutative, all positive maps (in particular, states) are automatically completely positive. Complete positivity is however a much more demanding property than just positivity. While the general structure of positive maps is rather elusive, even in a finite dimensional setting, the structure of CP maps is completely described by the Stinespring Theorem [Sti]. We may consider, without loss of generality, the case of a CP map $T:A\to B(h)$. The result ensures the existence of a representation $\pi:A\to B(k)$ on a Hilbert space $k$ and that of a bounded operator $V:h\to k$ such that
\[
Ta=V^*\pi(a)V\qquad a\in A.
\]
In case $A$ is unital and $T1_A=1_A$, then $V^*V=I_h$ so that $V$ is an isometry which can be considered as an immersion of $h$ into $k$. $V^*$ is then the projection of $k$ onto $h$ and the CP map $T$ results as the {\it compression of the restriction of a representation}. The Steinespring construction can be considered as a generalization of the GNS representation. One starts endowing the vector space $A\otimes_{\rm alg}h$ by the sesquilinear form
\[
(a\otimes \xi|b\otimes \eta):=(\xi|T(a^*b)\eta)_h\qquad a,b\in A\, ,\quad \xi,\eta\in h
\]
and checks that the CP property just ensures that this form is positive definite. Cutting out its kernel and completing the normed space obtained, one gets the Hilbert space $k$. The representation of $A$ on $k$ is an ampliation of the left regular representation of $A$ as it is induced by the map $a(b\otimes \xi)\mapsto ab\otimes\xi$.\\
A positivity preserving map $\phi:M\to N$ between von Neumann algebras, is {\it normal} if $\phi(\sup_\alpha x_\alpha)=\sup_\alpha \phi (x_\alpha)$ for all bounded monotone increasing nets of self adjoint elements $\{x_\alpha\}\subset M$. The property is equivalent to the continuity with respect to weak$^*$-topology of the algebras.
%The involution $a\mapsto a^*$ is an example of a positive map which is not completely positive.
\subsection{Positivity preserving and Markovian semigroups on operator algebras} ([Br1]).
A strongly continuous semigroup $\{T_t:t>0\}$ of contractions on a unital C$^*$-algebra $A$
\[
T_t:A\to A\qquad T_t\circ T_s=T_{t+s}\, ,\quad T_0=I\, ,\quad \lim_{t\to 0^+}\|a-T_t a\|_A=0\, ,\quad a\in A
\]
is said to be {\it Markovian} if it is {\it positivity preserving and subunital }
\[
0\le a\le 1_{A}\quad\Rightarrow\quad 0\le T_t a\le 1_{A} \qquad a\in A.
\]
If $A$ is endowed with a densely defined trace $\tau$, the semigroup is said to be $\tau$-symmetric if
\[
\tau(a^*(T_t b))=\tau((T_t a^*) b)\qquad a,b\in A\cap L^1(A,\tau).
\]
In case $A$ is a von Neumann algebra, one requires the trace to be normal and the semigroup to be point-weak*-continuous in the sense
\[
\lim_{t\to 0^+}\eta(a-T_t a)=0\qquad a\in A,\,\, \eta\in A_*.
\]
In case the C$^*$-algebra $A$ does not have a unit, one can understand positivity preserving and Markovianity embedding $A$ into a larger unital C$^*$-algebra $\widetilde A$ and there using the unit $1_{\widetilde A}$ instead of $1_A$. For example one can choose $A\oplus \mathbb{C}$.
\vskip0.1truecm\noindent
The generator $(L,D(L))$ of a Markovian semigroup on a C$^*$-algebra (resp. a von Neumann algebra) $A$ is a norm (resp. weak*) closed, densely defined operator on $A$ defined as
\[
D(L):=\{a\in A: \exists \lim_{t\to 0^+}\frac{a-T_t a}{t}\in A\}\qquad La:=\lim_{t\to 0^+}\frac{a-T_t a}{t}\quad a\in D(L)
\]
where the limit is understood in the norm (resp. weak$^*$)-topology. Norm continuous semigroups are exactly those which have bounded generators and these are classified in [Lin], [CE].
{\it Completely positive, completely contractive or completely Markovian semigroups} are defined as those semigroups on $A$ whose ampliations to the algebras $A\otimes \mathbb{M}_n(\mathbb{C})$ are positive, contractive or Markovian for all $n\ge 1$. Completely Markovian semigroups are also called {\it dynamical semigroups} especially in Mathematical Physics and Quantum Probability (see [D1]).
\begin{rem}
Notice that, on von Neumann algebras, {\it strongly continuous semigroups are automatically norm continuous} as it follows by a direct application of [E Theorem 1]. Since semigroups with bounded generators have rather limited applications, this is the reasons for which {\it on von Neumann algebra the natural continuity of a semigroups is the point-weak*-continuity}.
\end{rem}

\section{Noncommutative Potential Theory}
In this section, we let $(A,\tau)$ be a C$^*$-algebra endowed with a densely defined, lower semicontinuous faithful trace and consider the GNS representation $\pi_{GNS}$ acting on the space $L^2(A,\tau)$. We will indicate by $L^\infty(A,\tau)$ the von Neumann algebra $(\pi_{\rm GNS}(A))''\subseteq B(L^2(A,\tau))$ generated by $A$ through the GNS representation.\\
Recall that the little Lipschitz algebra is defined as
\[
{\rm Lip}_0(\R):=\{f:\R\to\R:f(0)=0, |f(t)-f(s)|\le|t-s|, t,s\in\R\}.
\]
If $a=a^*\in A$ and $f\in {\rm Lip}_0(\mathbb{R})$, then $f(a)\in A$ acquires a meaning thank to the fact that C$^*$-algebras are closed under continuous functional calculus. Since, by assumption, $A^\tau:=A\cap L^2(A,\tau)$ is dense in $A$ and a fortiori in $L^2(A,\tau)$, if $a=a^*\in L^2(A,\tau)$ then $f(a)\in L^2(A,\tau)$ may bedefined as the limit in $L^2(A,\tau)$ of the sequence $f(a_n)\in L^2(A,\tau)$ associated to a sequence $a_n\in A\cap L^2(A,\tau)$ converging to $a$ in $L^2(A,\tau)$.
\subsection{Dirichlet forms on C$^*$-algebras with trace d'apres Albeverio-Hoegh-Krohn}
In this section we define Dirichlet forms and Markovian semigroups on the space $L^2(A,\tau)$ and discuss the connection between them and the Markovian semigroups on the von Neumann algebra $L^\infty(A,\tau)$, where $(A,\tau)$ is a C$^*$-algebra $A$ endowed with a densely defined, lower semicontinuous faithful trace, introduced in [AHK1]. Even if we will not discuss them in this notes, we mention that D. Guido, T. Isola and S. Scarlatti in [GIS] provided the extension of this theory to the case of {\it non-symmetric Dirichlet forms}.
\begin{defn}
A {\it Dirichlet form} is a lower semicontinuous functional
\[
\E:L^2(A,\tau)\to (-\infty,+\infty]
\]
with domain $\F:=\{a\in L^2(A,\tau): \E[a]<+\infty\}$ satisfying the properties
\vskip0.1truecm\noindent
i) $\F$ is dense in $L^2(A,\tau)$\vskip0.2truecm\noindent
ii) $\E[a^*]=\E[a]$ for all $a\in L^2(A,\tau)$ (reality)\vskip0.2truecm\noindent
iii) $\E[f(a)]\le\E[a]$ for all $a=a^*\in L^2(A,\tau)$ and all $f\in {\rm Lip}_0(\R)$ (Markovianity).
\vskip0.2truecm\noindent
A Dirichlet form is said to be
\vskip0.1truecm\noindent
iv) {\it regular} if its domain $\F$ is dense in $A$\vskip0.2truecm\noindent
v) {\it complete Dirichlet form} if the ampliation $\E^n$ on the algebra $(A\otimes M_n(\mathbb{C}),\tau\otimes {\rm tr}_n)$ defined
\[
\E^n:L^2(A\otimes M_n(\mathbb{C}),\tau\otimes {\rm tr}_n)\to (-\infty,+\infty]\qquad \E^n[[a_{i,j}]_{i,j=1}^n]:=\sum_{i,j=1}^n\E[a_{i,j}]
\]
is a Dirichlet forms for all $n\ge1$.\\
\vskip0.2truecm\noindent
A strongly continuous, self-adjoint semigroup $\{T_t:t>0\}$ on $L^2(A,\tau)$ is said
\vskip0.1truecm\noindent
vi) {\it positivity preserving} if $T_ta\in L^2_+(A,\tau)$ for all $a\in L^2_+(A,\tau)$
\vskip0.1truecm\noindent

vii) {\it Markovian} if it is positivity preserving and for $a=a^*\in A\cap L^2(A,\tau)$
\[
0\le a\le 1_{\widetilde A}\quad \implies \quad 0\le T_t a\le 1_{\widetilde A}\qquad t>0
\]
viii) {\it completely Markovian} if the extensions $T^n_t:=T_t\otimes I_n$ to $L^2(A\otimes M_n(\mathbb{C}),\tau\otimes {\rm tr}_n)$ are Markovian semigroups for all $n\ge 1$.
\end{defn}
\begin{rem} 1) If in the Markovianity condition one considers as $f$ the zero function in ${\rm Lip}_0(\R)$, one verifies that Dirichlet forms are nonnegative.\\
2) It may be checked that Markovianity is equivalent to the single contraction property
\[
\E[a\wedge 1]\le\E[a]\qquad a=a^*\in L^2(A,\tau)
\]
in which only the {\it unit contraction} $f(t):=t\wedge 1$ is involved. A geometric Hilbertian interpretation of this fact will be vital to extend the theory beyond the trace case.\\
3) A nice characterization of elements of type $f(a)$ for a fixed $a=a^*\in L^2(A,\tau)$ and $f\in {\rm Lip}_0(\R)$ has been shown in [AHK1] as those hermitian $b=b^*\in L^2(A,\tau)$ such that
\[
b^2\le a^2\, ,\qquad |b\otimes 1-1\otimes b|^2\le |a\otimes 1-1\otimes a|^2\, .
\]
4) Since $L^2(A\otimes M_n(\mathbb{C}),\tau\otimes {\rm tr}_n)=L^2(A,\tau)\otimes L^2(M_n(\mathbb{C}),{\rm tr}_n)$, the ampliations are equivalently defined as
\[
\E[a\otimes m]:=\E[a]\cdot \|m\|^2_{\rm HS}\qquad a\otimes m\in L^2(A,\tau)\otimes L^2(M_n(\mathbb{C}),{\rm tr}_n)\, .
\]
\end{rem}
The first fundamental result of the Albeverio-Hoegh-Krohn work [AHK1] is the following correspondence which generalize that of Beurling-Deny in the commutative case.
\begin{thm}
There exists a one-to-one correspondence among
\vskip0.1truecm\noindent
i) Dirichlet forms $(\E,\F)$ on $L^2(A,\tau)$
\vskip0.1truecm\noindent
ii) Markovian semigroups $\{T_t:t>0\}$ on $L^2(A,\tau)$
\vskip0.1truecm\noindent
iii) $\tau$-symmetric, Markovian semigroups $\{S_t:t>0\}$ on the von Neumann algebra $L^\infty(A,\tau)$.
\vskip0.1truecm\noindent
Moreover, the semigroups are completely Markovian if and only if the quadratic form is a completely Dirichlet form.
\end{thm}\noindent
The correspondence between semigroups and quadratic forms on $L^2(A,\tau)$ is given by the relation
\[
\E[a]=\lim_{t\to 0}t^{-1}\bigl(a|(I-T_t)a\bigr)_{L^2(A,\tau)}\qquad a\in L^2(A,\tau)
\]
where both sides are finite precisely when $a\in\F$. The correspondence between the semigroup on the Hilbert space $L^2(A,\tau)$ and the one the von Neumann algebra $L^\infty(A,\tau)$ is given by
\[
S_t a=T_t a\qquad a\in A\cap L^2(A,\tau).
\]
\begin{rem}
The above correspondence is exactly the original one proved in [AHK1] even if the result still holds true if one start with a semi-finite von Neumann algebra $(M,\tau)$ and a densely defined, semifinite trace on it.\\
We prefer the first presentation since it prepares the ground i) to naturally introduce and discuss the notion of {\it regularity} of a noncommutative Dirichlet forms, which, as in the Beuling-Deny theory, is the key notion to develop a rich potential theory [CS4] and ii) to develop the intrinsic differential calculus of Dirichlet spaces (see Section 5 below and [CS1]).
\end{rem}
The second fundamental result of the Albeverio-Hoegh-Krohn work is the following
\begin{thm}
Let the C$^*$-algebra $A$ be represented as acting on a Hilbert space $h$. Let $K$ be a self-adjoint (non necessarily bounded) operator on $h$ and $m_i\in L^2(h)$ be Hilbert-Schmidt operators for $i=1, 2,\cdots$. Then the quadratic form
\[
\E[a]:=\sum_{i=1}^\infty {\rm Tr}(|[a,m_i]|^2)+{\rm Tr}(K|a|^2)\qquad a\in L^2(A,\tau)
\]
is a completely Dirichlet form provided it is densely defined.
\end{thm}
This result is fundamental not only because it provides a tool to construct a large class of examples but also because it suggests, at least in one direction, a correspondence between completely Dirichlet forms and unbounded derivations $a\mapsto i[a,m]$ on the C$^*$-algebra $A$ (see Section 5 below).
\vskip0.1truecm\noindent
The proofs in [AHK1] of both theorems are based on a careful analysis of the {normal contractions} on $A$.

\subsubsection{Dirichlet energy forms on Clifford C$^*$-algebras}
Here we illustrate the first example of a noncommutative Dirichlet form. It has been created to represents the quadratic form of a physical Hamiltonian of an assembly of electrons and positrons. In particular, its definition and the study of its properties has been introduced by L. Gross [G1,2] in connection with the problem of existence and uniqueness of the ground state of physical Hamiltonians describing Fermions.\\
Let $h$ be a complex Hilbert space and $J$ a conjugation on it (i.e. an anti-linear, anti-unitary operator such that $J^2=I$). Systems whose number of particles is not a priori bounded above are described by the Fock space
\[
\mathfrak{F}(h):=\bigoplus_{n=0}^\infty h^{\otimes^n}\, .
\]
Particles system obeying a Fermi-Dirac statistics are described by the Fermi-Fock subspace
\[
\mathfrak{F}_-(h):=P_-(\mathfrak{F}(h)).
\]
where the orthogonal projection $P_-$ is defined by
\[
P_-(f_1\otimes \cdots \otimes f_n)=(n!)^{-1}\sum_\pi \varepsilon_\pi f_{\pi(1)}\otimes \cdots \otimes f_{\pi(n)}
\]
where the sum is over all permutations $(1,\cdots ,n)\mapsto (\pi(1),\cdots ,\pi(n))$. For $f\in h$, the creation operator $a^*(f)$, defined as $a^*(f):=\sqrt{n+1}P_-(f\otimes g_1\otimes \cdots \otimes g_n)$, is bounded with norm $\|a^*(f)\|=\|f\|_h$. Together with the annihilation operator defined by $a(f):=(a^*(f))^*$, it satisfies the canonical anti-commutation CAR relations
\[
a(f)a(g)+a(g)a(f)=0,\qquad a(f)a^*(g)+a^*(g)a(f)=(f|g){\rm I}_h\qquad f,g\in h
\]
which represents the Pauli's Exclusion Principle. The Clifford C$^*$-algebra ${Cl}(h)$ is defined as the C$^*$-algebra generated by the fields operators
\[
b(f):=a^*(f)+a(Jf)\qquad f\in h
\]
i.e. as the intersection of all C$^*$-subalgebras of $B(\mathfrak{F}_-(h))$ containing the fields $\{b(f):f\in h\}$. It is highly noncommutative since it is a simple C$^*$-algebra in the sense that it has no nontrivial closed, bilateral ideals. The Fock vacuum vector $\O:=1\oplus 0\cdots \in \mathfrak{F}_-(h)$ defines a trace vector state on it by
\[
\tau_0(A):=(\O|A\O)\qquad A\in {Cl}(h)
\]
and the natural map $D:{Cl}(h)\to \mathfrak{F}_-(h)$ given by $A\mapsto A\O$, extends to a unitary map from $L^2({Cl}(h),\tau_0)$ onto $ \mathfrak{F}_-(h)$, called the Segal isomorphism. This natural isomorphism allows to transfer on the Fermi-Fock space the order structure one has on $L^2({Cl}(h),\tau_0)$ and viceversa, to study on $L^2({Cl}(h),\tau_0)$ operators originally created on $\mathfrak{F}_-(h)$. This procedure is especially useful in combination with second quantization, where a self-adjoint operator $(A,D(A))$ on $h$ give rise to a self-adjoint operator $(d\Gamma(A),D(d\Gamma(A)))$ on $\mathfrak{F}_-(h)$ as follows. First define self-adjoint operators $(A_n, D(A_n))$ on $P_-(h^{\otimes^n})$ for $n\ge 0$ setting $A_0=0$ and
\[
A_n(P_-(f_1\otimes\cdots\otimes f_n)):=\sum_{k=1}^n P_-(f_1\otimes\cdots\otimes Af_k\otimes\cdots f_n)\qquad f_1\otimes\cdots\otimes f_n\in D(A_n):=D(A)^{\otimes^n}.
\]
The direct sum of the $A_n$ is essentially self-adjoint because it is symmetric and it has a dense set of analytic vectors formed by finite sums of anti-symmetrized products of analytic vectors of $A$. The self-adjoint closure $d\Gamma(A):={\overline{\oplus_{k=0}^\infty A_n}}$ of this sum is called the {\it second quantization} of $A$ and is denoted by $(d\Gamma(A),D(d\Gamma(A)))$. The main example of this procedure concerns the Number operator $d\Gamma(I)$.
\begin{thm} (Clifford Dirichlet form)
Let $(A,D(A))$ be a self-adjoint operator on $h$, commuting with $J$ and satisfying $A\ge mI_h$ for some $m>0$. Then
\vskip0.1truecm\noindent
i) the quadratic form $(\E,\F)$ of the  operator $H:=D^{-1}d\Gamma(A)D$
\[
\E[\xi]:=(\xi|H\xi)_{L^2({Cl}(h),\tau_0)}=(D\xi|d\Gamma(A)D\xi)_{\mathfrak{F}_-(h)}\qquad \xi\in D(H)
\]
is a completely Dirichlet form on $L^2({Cl}(h),\tau_0)$;\\
ii) the completely Markovian semigroup $e^{-tH}$ is hypercontractive in the sense that it is bounded from $L^2({Cl}(h),\tau_0)$ to $L^4({Cl}(h),\tau_0)$ as soon as $mt\ge (\ln 3)/2$,\\
iii) $\inf\sigma(H)$ is an isolated eigenvalue of finite multiplicity.
\end{thm}
The main point is to prove the result for $A=I_h$ so that the Hamiltonian is $H=Dd\Gamma(I_h)D^{-1}=DND^{-1}$ is unitarely equivalent to the Number operator $N=d\Gamma(I_h)$. We shall see later a proof based on the structure of the Dirichlet form of the Number operator. The interest in noncommutative Dirichlet forms originated in QFT to extend to Fermions the non perturbative techniques of E. Nelson, I.M. Segal, J. Glimm, L. Gross, A. Jaffe, B. Simon and others, elaborated for Bosons systems.
\vskip0.2truecm\noindent
We conclude this section noticing that the {\it Clifford von Neumann algebra} $L^\infty({Cl}(h),\tau_0)$ generated by the GNS representation of the Clifford C$^*$-algebra provided by the Fock vacuum state, is isomorphic to the {\it the hyperfinite $II_1$-factor} (usually denoted by $R$) to which $\tau_0$ extends to a normal tracial state. While the {\it hyperfinitness} (see Section 7.3 below) is a reflection of the fact that $L^\infty({Cl}(h),\tau_0)$ is generated by the net of finite-dimensional subalgebras corresponding to finite dimensional subspaces of the Hilbert space $h$, its {\it uniqueness} is a fundamental result of A. Connes [Co 3]. As any von Neumann algebra, $R$ is generated by its projections $p\in R$ (which are defined as the self-adjoint elemets $p=p^*$ satisfying $p^2=p$). However, while projections in a type $I$ von Neumann algebra as $B(h)$ have traces which can assume integers values only (equal to the dimension of their ranges), the trace of a projection in $R$ may assume any real value $\tau_0(p)\in [0,1]$, interpreted as a {\it real dimension} of the range of $p$. This is the reason by which J. von Neumann regarded $R$ as exhibiting a {\it Euclidean continuous geometry}.

\section{Dirichlet forms and Differential Calculus: bimodules and derivations}
In this section we show that on C$^*$-algebras endowed with a densely defined, lower semicontinuous, faithful trace $(A,\tau)$, {\it completely Dirichlet forms are representations of a differential calculus} (see [CS1], [S1,2]). In fact they can be constructed, on one side, and determine, on the other side, closable derivations on the C$^*$-algebra $A$. This was suggested, at least in one direction, by the result of Albeverio-Hoegh-Krohn illustrated in Theorem 4.5 above.
\vskip0.2truecm\noindent
To specify what a derivation on a C$^*$-algebra $A$ is, let us recall the notion of $A$-bimodule $\H$: this is an Hilbert space together with two continuous commuting actions (say left and right) of $A$
\[
A\times\H\ni (a,\xi)\mapsto a\xi\in\H,\qquad \H\times A\ni (\xi,b)\mapsto \xi b\in\H\, .
\]
The commutativity says that $(a\xi)b=a(\xi b)$ for all $a,b\in A$ and $\xi\in\H$. If the left and right actions coincides, $a\xi=\xi a$, then $\H$ is called a $A$-mono-module. Equivalently, a $A$-bimodule is a representation on the Hilbert space $\H$ of the maximal or projective tensor product C$^*$-algebra $A\otimes_{\rm max}A^\circ$. Here $A^\circ$ denote the {\it opposite} C$^*$-algebra coinciding with $A$ as linear space with involution but in which the product is reversed in order: $x^\circ y^\circ:=(yx)^\circ$.\\
A {\it symmetric} $A$-bimodule $(\H,\J)$ is a $A$-bimodule $\J$ together with a conjugation $\J$ such that
\[
\J(a\xi b)=b^*(\J\xi)a^*\qquad a,b\in A,\quad \xi\in\H.
\]
\begin{defn}(Derivation on C$^*$-algebras)
A derivation on a C$^*$-algebra $A$ is a linear map $\partial:D(\partial)\to\H$ defined on a subalgebra $D(\partial)\subseteq A$ with values into a $A$-bimodule $\H$ satisfying the Leibnitz rule
\[
\partial (ab)=(\partial a)b+a(\partial b)\qquad a,b\in D(\partial)\subseteq A.
\]
The derivation is called symmetric if $D(\partial)$ is involutive, $\H$ is symmetric and
\[
\J(\partial a)=\partial a^*\qquad a\in D(\partial).
\]
\end{defn}
Here we review some examples of derivations.
\subsubsection{Gradient operator}
Let $M$ be a Riemannian manifold and consider the Hilbert space $L^2_{\mathbb{C}}(TM):=L^2(TM)\otimes_\R\mathbb{C}$ obtained complexifying the Hilbert space of square integrable vector fields. This is a mono-module\footnote{A mono-module is a bi-module in which the left and right actions coincide.} over the commutative C$^*$-algebra of continuous function $C_0(M)$ where the action is defined pointwise and it can be endowed with the involution $\J(\chi\otimes z):=\chi\otimes {\bar z}$. If $H^1(M)$ denotes the first Sobolev space, then a symmetric derivation is defined by the gradient operator
\[
\nabla : C_0(M)\cap H^1(M)\to L^2_{\mathbb{C}}(TM).
\]
\subsubsection{Difference operator}
Let $X$ be a locally compact Hausdorff space and let $j$ be a Radon measure on $X\times X$ supported off the diagonal. Left and right commuting actions of $C_0(X)$ on $L^2(X\times X, j)$ may be defined as
\[
(af)(x,y):=a(x)f(x,y),\qquad (fb)(x,y)=f(x,y)b(y)\qquad a,b\in C_0(X),\quad f\in L^2(X\times X,j)
\]
and one may check that
\[
j:C_c(X)\to L^2(X\times X,j)\qquad (\partial_ja)(x,y):=a(x)-a(y)
\]
is a symmetric derivation on $C_0(X)$ once the conjugation is defined as
\[
(\J f)(x,y):={\overline{f(y,x)}}.
\]
\subsubsection{Killing measure}
Let $X$ be a locally compact Hausdorff space and let $k$ be a Radon measure on $X$. Consider $L^2(X,k)$ as a $C_0(X)$-bimodule where the left action is the usual pointwise one while the right action is the trivial one so that $\xi b:=0$ for all $\xi\in L^2(X,k)$ and $b\in C_0(X)$. If one considers as $\J$ just the pointwise complex conjugation of functions in $L^2(X,k)$, then one may easily check that the map
\[
\partial_k:C_c(X)\to L^2(X,k)\qquad \partial a:=a
\]
is a symmetric derivation on $C_0(X)$.
\subsubsection{Commutators I}
Let $(A,\tau)$ be a C$^*$-algebra endowed with a faithful, semifinite trace and recall that $A_\tau:=\{a\in A:\tau(a^* a)<+\infty\}$ is a bilateral ideal in $A$. Then if one consider on the Hilbert space $L^2(A,\tau)$ the natural left and right actions of $A$ and the conjugation $\J a:=a^*$, one obtains a symmetric $A$-bimodule. Moreover any $b\in A $ give rise to a symmetric derivation
\[
\partial_b:A_\tau\to L^2(A,\tau)\qquad \partial_b a:=i[a,b]=i(ab-ba).
\]
If a sequence $\{b_k\in A:k\ge 1\}$ is fixed and one consider the direct sum of symmetric $A$-bimodules $\oplus_{k=1}^\infty L^2(A,\tau)$, then the direct sum
\[
\partial :=\oplus_{k=1}^\infty \partial_{b_k} :D(\partial)\to \oplus_{k=1}^\infty L^2(A,\tau)
\]
is a symmetric derivation defined on the involutive subalgebra $D(\partial)$ of those $a\in A_\tau$ such that the series $\oplus_{k_1}^\infty \|[a,b_k]\|^2_{L^2(A,\tau)}$  converges.
\subsubsection{Commutators II}
As a variation of the above construction, suppose that $A$ is represented on the Hilbert space $h$. Then the space of Hilbert-Schmidt operators $L^2(h)$ is a $A$-bimodule for the left $a\xi$ and right $\xi b$ actions of $a,b\in A$ on $\xi\in L^2(h)$ given by composition as operators on $h$. A natural involution is defined by $\J\xi:=\xi^*$ on $\xi\in L^2(h)$ and then a symmetric derivation is given by
\[
\partial_\xi:A\to L^2(h)\qquad \partial_\xi a:=i(a\xi-\xi a)\, .
\]
If a sequence $\{\xi_k\in L^2(h):k\ge 1\}$ of Hilbert-Schmidt operators is fixed and one consider the direct sum of symmetric $A$-bimodules
$\oplus_{k=1}^\infty L^2(h)$, then the direct sum
\[
\partial :=\oplus_{k=1}^\infty \partial_{\xi_k} :A\to \oplus_{k=1}^\infty L^2(h)
\]
is a symmetric derivation defined on the involutive subalgebra $D(\partial)$ of those $a\in A_\tau$ such that the series $\oplus_{k=1}^\infty \|a\xi_k -\xi_k a\|^2_{L^2(h)}$  converges. This last derivation is clearly related to the one appearing in the Albeverio-Hoegh-Khron Theorem above. We will come back on this construction later on.
\vskip0.2truecm
Next result shows that closable derivations give rise to Dirichlet forms ([CS1]).
\begin{thm}
Let $(A,\tau)$ be a C$^*$-algebra endowed with a densely defined, lower semicontinuous, faithful trace and let $(\partial,D(\partial),\H,\J)$ be a symmetric derivation, densely defined on a domain $D(\partial)\subset A\cap L^2(A,\tau)$, which is closable as an operator from $L^2(A,\tau)$ to $\H$. Then the closure of the quadratic form
\[
\E[a]:=\|\partial a\|^2_\H\qquad a\in \F:=D(\partial)
\]
is a completely Dirichlet form.
\end{thm}\noindent
The proof of the above result goes through the establishment of {\it noncommutative chain rule} [CS1] for closable derivation, by which one has
\[
\partial f(a)=((L_a\otimes R_a)(\tilde f))(\partial a)\qquad a=a^*\in A\cap L^2(A,\tau),\quad f\in {\rm Lip}_0(\R)
\]
Here $L_a$ (resp. $R_a$) are the representation of ${\rm C} (sp(a))$, continuous, complex valued functions on the spectrum $sp(a)$ of
$a$, uniquely defined by
\[
\begin{aligned}
L_a(f)\xi=
\begin{cases}
& f(a)\xi \text{ if } f(0)=0 \\ &\xi \text{ if }f\equiv 1
\end{cases}
\quad f\in C(sp(a)) \quad \xi\in \mathcal{H}
\end{aligned}
\]
and
\[
\begin{aligned}
R_a(f)\xi=
\begin{cases}
& \xi f(a)\text{ if } f(0)=0 \\ & \xi \text{ if }f\equiv 1
\end{cases}
\quad f\in C(sp(a)) \quad \xi\in \mathcal{H} \, .
\end{aligned}
\]
$L_a\otimes R_a$ is the tensor product representation of $C(sp(a))\otimes C(sp(a))=C(sp(a)\times sp(a))$. When $I\subseteq \mathbb{R}$ is a closed interval and $f\in {\rm C}^1 (I)$, we will denote by $\tilde{f}\in {\rm C}(I\times I )$ the {\it differential quotient} on $I\times I$, sometimes called the {\it quantum derivative} of $f$, defined by
\begin{equation}
\begin{aligned}
\tilde{f}(s,t)=
\begin{cases}
& \frac{f(s)-f(t)}{s-t}\,\,\, \text{ if }\,\, s\neq t \\ & f'(s)\qquad
\text{ if }\,\, s=t.
\end{cases}
\end{aligned}
\end{equation}
Since commutators in $A$ are bounded derivations in the above sense, the above result provides an independent proof of the Albeverio-Hoegh-Khron Theorem 4.5 above.

\subsubsection{A derivation for the Clifford-Dirichlet form}
As a further application of the above result, let us show that the quadratic form $\E_N$ of the Number operator $N$ of Fermions, when seen on the space $L^2({Cl}(h),\tau_0)$ via the Segal isomorphism $D$, is a completely Dirichlet form on the Clifford algebra $({Cl}(h),\tau_0)$.
Recall first that $N$ and $\E_N$ can be written as
\[
N=\sum_{k=1}^\infty a^*(f_k)a(f_k)\qquad \E_N[\psi]:=\sum_{k=1}^\infty \|a(f_k)\psi\|^2
\]
for any orthonormal base of $\{f_k:k\ge 1\}\subset h$. Let us denote by $L_b$ and $R_b$ the left and right actions on $L^2({Cl}(h),\tau_0)$ of an element of the Clifford algebra $b\in {Cl}(h)$. The symmetry $\beta:=\Gamma(-I_h)$ of $\mathfrak{F}_-(h)$ induces an idempotent automorphism $\gamma\in B(\mathfrak{F}_-(h))$
\[
\gamma(A):=\beta A\beta\qquad A\in B(\mathfrak{F}_-(h)).
\]
Since $b(f)\beta=-b(f)$, we have $\gamma(b(f))=-b(f)$ for all $f\in h$ so that $\gamma$ leaves ${Cl}(h)$ globally invariant and then $\gamma\in {\rm Aut}({Cl}(h))$. Since ${Cl}(h)\subset L^2({Cl}(h),\tau_0)$ and $\beta\O=\O$ we have $D(\gamma(b))=\beta b\beta\O=\beta b\O=\beta D(b)$ so that
\[
\gamma(b)=(D^{-1}\beta D)(b)\qquad b\in {Cl}(h).
\]
For $f\in h$ and $\xi\in {Cl}(h)$ we have $DL_{b(f)} D^{-1}(D\xi)=D(L_{b(f)}\xi)=L_{b(f)}(\xi\O)=L_{b(f)}(D\xi)$ so that, since $ {Cl}(h)$ is dense in $L^2({Cl}(h),\tau_0)$, we obtain
\[
DL_{b(f)} D^{-1}=a^*(f)+a(Jf)=b(f).
\]
Since $\beta(P_-(f_1\otimes\cdots\otimes f_n))=P_-\beta(f_1\otimes\cdots\otimes f_n)=(-1)^nP_-(f_1\otimes\cdots\otimes f_n)$ and $\beta\O=\O$, by the CAR relations we have that $a^*(f)-a(Jf)\beta$ commutes with all of $b(g)$ for $g\in h$ and then with all elements of the Clifford algebra. Then, for $b\in {Cl}(h)$ we have $(a^*(f)-a(Jf))\beta(D(b))=(a^*(f)-a(Jf))\beta(b\O)=b(a^*(f)-a(Jf))\beta\O=b(a^*(f)-a(Jf))\O=ba^*(f)\O=b(a^*(f)+a(Jf))\O=bb(f)\O=(R_{b(f)}b)\O=D(R_{b(f)}b)$ and since $ {Cl}(h)$ is dense in $L^2({Cl}(h),\tau_0)$, we obtain
\[
DR_{b(f)} D^{-1}=a^*(f)-a(Jf)\beta
\]
which can be rewritten as
\[
a^*(f)-a(Jf)=DR_{b(f)} D^{-1}\beta=DR_{b(f)} (D^{-1}\beta D)D^{-1}.
\]
By summation we have
\[
a(Jf)=\frac{1}{2}D(L_{b(f)}-R_{b(f)}(D^{-1}\beta D))D^{-1}
\]
and changing $f$ with $Jf$ we get
\[
a(f)=D\Bigl(\frac{1}{2}(L_{b(Jf)}-R_{b(Jf)}(D^{-1}\beta D)\Bigr)D^{-1}.
\]
Consider now on the Hilbert space $\H_\gamma :=L^2({Cl}(h),\tau_0)$ the ${Cl}(h)$-bimodule structure where the right action of ${Cl}(h)$ is the usual one while the left one is twisted by the automorphism $\gamma$
\[
b_1\cdot\xi\cdot b_2:=\gamma (b_1)\xi b_2\qquad b_1,b_2\in{Cl}(h),\quad\xi\in L^2({Cl}(h),\tau_0)
\]
(dots indicate actions in this new bimodule structure). The definition is justified by the fact that, as one can easily check, this new left action is continuous and commutes with the right one. Let us now check that for any fixed $f\in h$, the map
\[
\partial_f:{Cl}(h)\to \H_\gamma\qquad \partial_f :=\frac{i}{2}(L_{b(Jf)}-R_{b(Jf)}(D^{-1}\beta D)),
\]
more explicitly given by a module commutator
\[
\begin{split}
\partial_f (b)&=\frac{1}{2}(L_{b(Jf)}-R_{b(Jf)}(D^{-1}\beta D))(b)\\
&=\frac{1}{2}(L_{b(Jf)}(b)-R_{b(Jf)}(D^{-1}\beta D)(b))\\
&=\frac{1}{2}(b(Jf)b-R_{b(Jf)}(\gamma(b))\\
&=\frac{1}{2}(b(Jf)b-\gamma(b)b(Jf))\\
&=\frac{1}{2}(b(Jf)\cdot b-b\cdot b(Jf))\\
\end{split}
\]
is a derivation in the sense that the Leibniz rule holds true:
\[
\begin{split}
\partial_f (ab)&=\frac{1}{2}\{b(Jf)\cdot(ab)-(ab)\cdot b(Jf)\}\\
&=\frac{1}{2}\{(b(Jf)\cdot a-a\cdot b(Jf))\cdot b+a\cdot (b(Jf)\cdot b-b\cdot b(Jf))\}\\
&=(\partial_f a)\cdot b+a\cdot (\partial_fb)\qquad a,b\in {Cl}(h).
\end{split}
\]
To define a symmetry $\J_\gamma$ such that $(\H_\gamma,\J_\gamma)$ is a symmetric bimodule, notice first that the automorphism $\gamma$ of the Clifford algebra ${Cl}(h)$ can be extended to a symmetry acting on the Hilbert space $L^2({Cl}(h),\tau_0)$. In fact as $\beta\O=\O$, the trace $\tau_0$ is $\gamma$-invariant: $\tau_0(\gamma(b))=(\O|\gamma(b)\O)=(\beta\O|b\beta\O)=(\O|b\O)=\tau_0(b)$ for all $b\in {Cl}(h)$. Consequently $\gamma$ is isometric with respect to $L^2$-norm
\[
\|\gamma(b)\|^2_2=\tau_0((\gamma(b))^*\gamma(b))=\tau_0(\gamma(b^*)\gamma(b))=\tau_0(\gamma(b^*b))=\tau_0(b^*b)=\|b\|^2_2\qquad b\in {Cl}(h)
\]
so that by density it extends to an isometry on the whole $L^2({Cl}(h),\tau_0)$ such that $\gamma^2=I$. Further, as for all $a,b\in {Cl}(h)$ we have
\[
(\gamma(a)|b)_2=\tau_0((\gamma(a))^*b)=\tau_0((\gamma(a^*))b)=\tau_0(\gamma(a^*\gamma(b)))=\tau_0(a^*\gamma(a^*))=(a|\gamma(b))_2 ,
\]
it results that $\gamma$ is a symmetry on $L^2({Cl}(h),\tau_0)$ which commutes with the symmetry $\J$
\[
\gamma^*=\gamma\, ,\quad \gamma^2=I\, ,\quad\gamma\circ\J=\J\circ\gamma\, .
\]
A new conjugation is then defined by $\J_\gamma:=\J\circ\gamma = \gamma\circ\J$ on $\H_\gamma$ in such a way that $(\H_\gamma,\J_\gamma)$ is a symmetric bimodule over the Clifford algebra ${Cl}(h)$, as it results from the following identities for $a,b,c\in {Cl}(h)$
\[
\J_\gamma(a\cdot b\cdot c)=\J(\gamma(\gamma(a)b c))=\J(a\gamma(b)\gamma(c)))=(\gamma(c))^*(\gamma(b))^*a^*=(\gamma(c^*))(\J_\gamma(b)a^*=c^*\cdot \J_\gamma(b)\cdot a^*\, .
\]
Since for $f\in h$ one has $\J_\gamma(b(Jf))=\gamma(b(Jf)^*)=\gamma(b(f))=-b(f)$, it follows that
\[
\begin{split}
\J_\gamma(\partial_f b)&=\frac{1}{2}\J_\gamma((b(Jf)\cdot b-b\cdot b(Jf)))\\
&=\frac{1}{2}(b^*\cdot\J_\gamma(b(Jf))-\J_\gamma (b(Jf))\cdot b^* )\\
&=\frac{1}{2}(b(f)\cdot b^*-b^* b(f))\\
&=\partial_{Jf}(b^*).
\end{split}
\]
Consequently, if $f\in h$ is $J$-real (in the sense that $Jf=f$) then $\partial_f$ is $\J_\gamma$-symmetric
\[
\J_\gamma(\partial_f b)=\partial_{J}(b^*)\qquad b\in  {Cl}(h).
\]
Choosing a Hilbert base $\{f_k:k\ge 1\}\subset h$ made by $J$-real vectors $Jf_k=f_k$, we can represent the quadratic form of the Number operator on the space $L^2( {Cl}(h),\tau_0)$ as
\[
\E_{Cl}[b]:=\sum_{k=1}^\infty \|\partial_{f_k} b\|^2_{\H_\gamma}\qquad b\in  \F_{Cl}
\]
where the form domain is obviously $\F_{Cl}:=D^{-1} (D(\sqrt N))$. Setting $\H_{Cl}:=\bigoplus_{k=1}^\infty \H_\gamma$ as a direct sum of symmetric ${Cl}(h)$-bimodules, we have that  $\partial_{Cl}:=\bigoplus_{k=1}^\infty \partial_{f_k}$ is a symmetric derivation of the Clifford algebra into $\H_{Cl}$ such that
\[
\E_{Cl}[b]=\|\partial_{Cl} b\|^2_{\H_{Cl}}\qquad b\in\F_{Cl}
\]
and which is densely defined and closable as sum of bounded derivations. By Theorem 5.2 above, the associated semigroup is thus completely Markovian.

\subsubsection{Dirichlet form on noncommutative tori.}
This is a fundamental example appearing in Noncommutative Geometry [Co5] in which the relevant {\it algebra of coordinates} $A_\theta$ of the space is a noncommutative deformation of the algebra of continuous functions on a classical torus. For any fixed $\theta\in [0,1]$, $A_\theta$, called {\it noncommutative $2$-torus}, is defined as the
universal C$^*$-algebra generated by two unitaries $U$ and $V$,
satisfying the relation
\[
VU=e^{2i\pi\theta}UV\, .
\]
When $\theta=0$ one recovers the algebra of continuous functions on the 2-torus. All elements of $A_\theta$ can be written as a series $\sum_{m,n\in\mathbb{Z}}c_{m,n}U^mV^n$  with complex coefficients.
A tracial state is specified by
\[
\tau : A_\theta \to\C\qquad\tau(U^m V^n)=\delta_{m,0}\delta_{n,0} \qquad m,n\in \mathbb{Z}
\]
so that
\[
L^2(A_\theta,\tau))=\Bigl\{\sum_{m,n\in\mathbb{Z}}c_{m,n}U^mV^n:\sum_{m,n\in\mathbb{Z}}|c_{m,n}|^2<+\infty\Bigr\}\simeq l^2(\mathbb{Z}^2).
\]
A densely defined, closed form is given by
\[
\E\Bigl[\sum_{m,n\in\mathbb{Z}}c_{m,n}U^mV^n\Bigr] =\sum_{m,n\in\mathbb{Z}}(m^2+n^2)|c_{m,n}|^2
\]
on the domain $\F\subset l^2(\mathbb{Z}^2)$ where the above series converges (i.e. the first Sobolev space). To check that we are dealing with a Dirichlet form, one may observe that it is a "square of a derivation", taking values in the direct sum of two standard bimodules $L^2(A,\tau)\oplus L^2(A,\tau)$ and given by the direct sum
\[
\partial(a)=\partial_1(a)\oplus \partial_2(a)
\]
of the derivations $\partial_1$ and $\partial_2$ from $A_\theta$ into $L^2(A_\theta,\tau)$ defined by
\[
\partial_1(U^mV^n)=imU^mV^n \; , \qquad\partial_1(U^mV^n)=inU^mV^n\qquad n,m\in\mathbb{Z}\, .
\]
The associated Markovian semigroup $\{T_t :t\ge 0\}$, characterized by
\[
T_t(U^mV^n)= e^{-t(m^2+n^2)}U^mV^n \qquad m,n\in\mathbb{Z},
\]
is clearly conservative, in the sense that $T_t 1_{A_\theta}=1_{A_\theta}$, because $\E[1_{A_\theta}]=0$. Even if, at a Hilbert space level, the Dirichlet form and its associated Markovian semigroup are clearly isomorphic to the Dirichlet integral and the heat semigroup on the classical (commutative) torus $\mathbb{T}^2$, written in Fourier series terms, their potential theoretic properties arise from the order structure of $A_\theta$ which may differ completely from those of $C(\mathbb{T}^2)$. The properties of $A_\theta$ depend a lot upon the rationality/irrationality and diophantine approximation properties of the parameter $\theta\in [0,1]$. For example consider the spectrum of the self-adjoint element $S:=U+U^*+V+V^*\in A_\theta$. If $\theta=0$ we have $A_\theta=C(\mathbb{T}^2)$ so that since $\mathbb{T}^2$ is connected and $S$ a real continuous function, its spectrum is a closed interval of the real line. When $\theta$ is irrational, however, the spectrum of $S$ is typically a Cantor set (so that, under "commutative spectacles" we would look at $A_\theta$ as a rather fragmented space).

\subsection{The derivation determined by a Dirichlet form}
The following result, in combination with the previous one, establishes a one-to-one correspondence between Dirichlet forms and closable derivations on C$^*$-algebras ([CS1], [S1,2]). It says that derivations are differential square roots of Dirichlet forms. It can be considered as a generalization of the construction of the (differential first order) Dirac operator from the (differential second order) Hodge-de Rham Laplacian of a Riemannian manifold.
\begin{thm}
Let $(\E,\F)$ be a completely Dirichlet form on a C$^*$-algebra endowed with a densely defined, lower semicontinuous faithful trace $(A,\tau)$\
Then there exists a densely defined, $L^2$-closable derivation $(\partial, D(\partial))$ with values in a symmetric $A$-bimodule  $(\H,\J)$ such that $D(\partial)$ is a form core and
\[
\E[a]=\|\partial a\|^2_\H\qquad a\in D(\partial).
\]
\end{thm}
The bimodule $(\H,\J)$ is called the {\it tangent bimodule} associated to $(\E,\F)$.

\subsubsection{Conditionally negative definite functions and Dirichlet forms on dual of discrete groups [CS1]}
Let $\lambda:\Gamma\to B(^2(\Gamma))$ be the left regular representation of a countable discrete group $\Gamma$
\[
(\lambda_sf)(t):=f(s^{-1}t)\qquad s,t\in\Gamma
\]
and consider the reduced C$^*$-algebra $C^*_r(\Gamma)$ defined as the smallest C$^*$-subalgebra of $B(L^2(\Gamma))$ containing the unitary operators $\lambda_s$ for all $s\in \Gamma$. Extending $\lambda$ to $c_c(\Gamma)$ as a convolution
\[
(\lambda(f)g)(t):=(f\ast g)(t)=\sum_{s\in\Gamma}f(ts^{-1})g(s)\qquad f,g\in c_c(\Gamma),\qquad t\in\Gamma ,
\]
we can identify $c_c(\Gamma)$ as a dense involutive subalgebra of $C^*_r(\Gamma)$. The involution of an element $f\in c_c(\Gamma)$ is given by $f^*(s)=\overline {f(s^{-1})}$. A faithful tracial state is determined  by
\[
\tau(\lambda(f))=f(e)\qquad f\in c_c(\Gamma)
\]
where $e\in\Gamma$ is the unit of the group. Since
\[
(\lambda(f)|\lambda(g))_{L^2(C^*_r(\Gamma),\tau)}=
%\tau (\lambda(f^*)\lambda(g))=\tau(\lambda(f^*\ast g))=(f^*\ast g)(e)=\sum_{s\in\Gamma}f^*(s^{-1})g(s)=\sum_{s\in\Gamma}{\overline {f(s)}}g(s)=
(f|g)_{l^2(\Gamma)}\qquad f,g\in c_c(\Gamma),
\]
we can identify $L^2(C^*_r(\Gamma),\tau)$ with $l^2(\Gamma)$ at the Hilbert space level and represent the positive cone $L^2_+(C^*_r(\Gamma),\tau)$ as the one of all square integrable, positive definite functions on $\Gamma$. The von Neumann algebra $L(\Gamma)$ generated by the unitaries  $\{\lambda_s:s\in \Gamma\}$ on $l^2(\Gamma)$ is called the {\it group von Neumann algebra} of $\Gamma$. It is a finite von Neumann algebra since the tracial state $\tau$ on $C^*_r(\Gamma)$ extends to a normal tracial state on it and it is a factor if and only if the conjugacy class
\[
\{sts^{-1}\in\Gamma:s\in\Gamma\}
\]
of any $t\in\Gamma$ is an infinite set. In this setting, any {\it conditionally negative definite function} $l:\Gamma\to \mathbb{C}$, i.e. a normalized, symmetric function
\[
l(e)=0,\qquad l(s^{-1})={\overline {l(s)}}\qquad s\in G
\]
satisfying, for $s_1,\dots,s_n\in\Gamma$, $c_1,\dots,c_n\in\mathbb{C}$,
\[
\sum_{k=1}^n \overline{c_j} l(s_j^{-1} s_k)c_k \le 0\qquad{\rm whenever}\qquad \sum_{k=1}^n c_k=0,
\]
determines a completely Dirichlet form
\[
\E_l[a]:=\sum_{s\in\Gamma} l(s)|a(s)|^2\qquad a\in l^2(\Gamma)
\]
defined on the domain $\F_l$ where the above sum converges and whose associated completely Markovian semigroup is given by
\[
(T_t a)(s)=e^{-tl(s)}a(s)\qquad a\in l^2(\Gamma).
\]
To describe the associated derivation recall that any negative definite function can be represent by a $1$-cocycle $c:\Gamma\to \H_\pi$ of a unitary representation
$\pi:\Gamma\to\H_\pi$, i.e. a function satisfying
\[
c(st)=c(s)+\pi(s)(c(t))\qquad s,t\in\Gamma ,
\]
as follows
\[
l(s)=\|c(s)\|^2_{\H_\pi}\qquad s\in\Gamma.
\]
On the Hilbert space $\H_\pi \otimes l^2(\Gamma)$ a $C_r^*(\Gamma)$-bimodule structure is then specified by the left action $\pi\otimes\lambda$ and by the right action $id\otimes\rho$ where $\rho$ is the right regular action of $\Gamma$ and $id$ is the trivial action on $\H_\pi$ assigning the identity operator to any element of $\Gamma$. Using the natural isomorphism $\H_\pi\otimes l^2(\Gamma)\simeq l^2(\Gamma,\H_\pi)$, the derivation representing $\E_l$ is identified by
\[
(\partial_l a)(s)=a(s)c(s)\qquad a\in c_c(\Gamma),\quad s\in\Gamma.
\]
When $\Gamma=\mathbb{Z}^n$, the C$^*$-algebra $C_r^*(\Gamma)$ coincides with the algebra $C(\mathbb{T}^n)$ of continuous functions on the torus and if one considers the negative definite function $l(z_1,\cdots,z_n):=|z_1|^2+\cdots + |z_n|^2$, one recovers as Dirichlet form just the standard Dirichlet integral on $\mathbb{T}^n$ (see Section 2.1).

\subsection{Decomposition of derivation,  Beurling-Deny fromula revisited}([CS1,2]).
\vskip0.1truecm
Since an $A$-bimodule is just a representation of the C$^*$-algebra $A\otimes_{\rm max} A$, one disposes of all the tools that representation theory offers, such has decomposition theory, to analyze derivations and their associated Dirichlet forms. In the commutative situation, for example, one obtains, by an algebraic approach, a refinement of the Beurling-Deny decomposition of Dirichlet forms.
\vskip0.2truecm\noindent
In this section we introduce notions of bounded, approximately bounded and completely unbounded derivations and we prove that any derivation canonically split into a sum of the latter.\\
Let $\partial:D(\partial)\to \H$ be a densely defined derivation on a C$^*$-algebra $A$ and denote by $\mathcal{L}_A(\H)$ the algebra of bounded operators on $\H$ which commute both with left and right actions of $A$.\\
 An element $B\in\mathcal{L}_A(\H)$ will be said to be $\partial$-bounded if the map $B\circ\partial$ extends to a bounded map from $A$ into $\H$. Notice that if this is the case, $B\circ\partial$ is a derivation. A projection $p\in\mathcal{L}_A(\H)$ will be said to be approximately $\partial$-bounded if it is the increasing limit of a net of $\partial$-bounded projections.  As $\H$ is assumed to be separable, this means that one can write the $A$-bimodule $p(\H)$ as an at most countable direct sum $\bigoplus_n\H_n$ of $A$-bimodules
such that $p\circ\partial$ decomposes as a direct sum $\bigoplus_n \partial_n$ of bounded derivations $\partial_n:=p_n\circ\partial$ where $p_n\in\mathcal{L}_A(\H)$ is the $\partial$-bounded projection onto the $A$-submodule $\H_n$. A projection $p\in\mathcal{L}_A(\H)$ will be said to be completely $\partial$-unbounded
if $0$ is the only $\partial$-bounded projection smaller than $p$. The derivation $\partial$ will be said to be bounded (resp. approximately bounded, resp. completely unbounded) if the identity operator $1_\H$ is a $\partial$-bounded (resp. approximately $\partial$-bounded, resp. completely $\partial$-unbounded) projection.\\
Then one can prove that there exists a greatest approximately $\partial$-bounded projection $P_a\in\mathcal{L}_A(\H)$ and that every $\partial$-bounded $B\in\mathcal{L}_A(\H)$ satisfies $B\circ P_a = B$.\\
Setting $\H_a:=P_a(\H)$ and $\H_c:=(1-P_a)(\H)$  we have the decomposition of $\H=\H_a\oplus\H_c$ into its approximately bounded and completely unbounded sub-$A$-bimodules. Correspondingly, setting $\partial_a:=P_a\circ\partial$ and $\partial_c:=(1-P_a)\circ\partial$ we have the decomposition of the derivation $\partial=\partial_a\oplus\partial_c$ into its approximately bounded and completely unbounded components. Finally, any Dirichlet form can be canonically splitted as a sum of its approximately bounded and completely unbounded parts
\[
\E[a]=\|\partial a\|^2_\H=\|\partial_a a\|^2_{\H_j} + \|\partial_c a\|^2_{\H_c}\qquad a\in \F.
\]
This is a (purely algebraic) generalization of the Beurling-Deny decomposition of Dirichlet forms on a commutative C$^*$-algebra $A=C_0(X)$. In particular the completely unbounded part $\E_c$ can be identified with the diffusive part and the approximately bounded part $\E_a$ correspond to the sum $\E_j+\E_k$ of the jumping and killing parts. In the commutative situation $\E_c$ can also be characterized as the part of the Dirichlet form whose $C_0(X)$-bimodule $\H_c$ is the largest sub-$C_0(X)$-mono-module of the $C_0(X)$-bimodule $\H$ corresponding to $\E$, i.e. the largest submodule on which the left and right actions coincide. Moreover, since as any $C_0(X)$-mono-module, $\H_c$ is the direct integral $\int_X \H_x\,\mu(dx)$ of $C_0(X)$-mono-modules $\H_x$ whose actions are the simplest possible
\[
a\xi=\xi a=a(x)\xi\qquad a\in C_0(X),\quad \xi\in\H_x,\quad x\in X,
\]
in the corresponding splitting $\partial=\int_X \partial_x \mu(dx)$, the derivations $\partial_x$ of $C_0(X)$ satisfy the Leibniz property
\[
\partial_x (ab) = (\partial_ x a)b(x)+a(x)(\partial_x b)\qquad a,b\in\F\cap C_0(X).
\]

%\subsection{Foliations} We can look at a foliation in two fundamental directions, the longitudinal transverse ones.
\subsection{Noncommutative potential theory and curvature in Riemannian Geometry}([CS2]).
Classical Potential Theory arose to understand properties of the potential energy functions in electromagnetism and in classical gravity. The properties of these functions were encoded in properties of the Laplace-Beltarmi operators and in those of the Dirichlet integrals of Euclidean domains. Dealing with Nonlinear Elasticity or Riemannian Geometry, one is naturally lead to consider other Laplace type operators and associated quadratic energy forms acting on sections of vector bundles over Riemannian manifolds. In this section we describe briefly the strict relation between curvature and a distinguished noncommutative Dirichlet form.
\vskip0.2truecm
Topological and geometric aspects of a Riemannian manifold $(V,g)$ are related to the Hodge-de Rham operator $\Delta_{\rm HdR}=dd^*+d^* d$ acting on the space $L^2(\Lambda^*(V))$ of square integrable sections of the exterior bundle $\Lambda^*(V)$. It generalizes the Laplace-Beltrami operator acting on functions but its quadratic form cannot be directly considered as a Dirichlet form, essentially because exterior forms do not realize a C$^*$-algebra. However, the  geometric aspects of $V$ are more deeply connected to the Dirac operator $D$ and its square $D^2$, the so called Dirac Laplacian, acting on sections of the Clifford bundle $Cl(V)$ essentially because it is the construction of this space and operators that depends on the metric $g$.\\
Recall that the fibers of $Cl(V)$ are the Clifford algebras $Cl(T_xV)$ of the Hilbert space $(T_xV,g_x)$. Since the exterior algebra $\Lambda^*_x V$ is nothing but the antisymmetric Fock space $\mathfrak{F}_-(T_xV)$, globalizing the Segal isomorphism, we met in a previous section, we have a canonical isomorphism of vector bundles between $Cl(V)$ and $\Lambda^*(V)$. The difference is that, by construction, the fibers of the bundle $Cl(V)$ form now C$^*$-algebras. As a consequence, the space $C^*_0(V,g)$ of continuous sections vanishing at infinity of the Clifford bundle form, by pointwise product on $V$, a C$^*$-algebra naturally associated to the Riemannian manifold $(V,g)$. Moreover, denoting by $\O_x\in \mathfrak{F}_-(T_xV)$ the vacuum vector and by $\tau_x (\cdot)=(\O_x|\cdot\O_x)$ the associated trace on $Cl(T_xV)$, using the Riemannian measure $m_g$, we get on the Clifford $^*$-algebra a densely defined, lower semicontinus, faithful trace
\[
\tau(\o):=\int_V  m_g(dx)\tau_x(\o_x)
\]
and the ordered Hilbert space $L^2(Cl(V,g),\tau)$. The Levi-Civita connection of $(V,g)$ can be lifted to a metric connection on the Clifford bundle and  represented, at the analytical level, by the covariant derivative $\nabla$ acting between the smooth sections of the Hermitian bundles $Cl(V,g)$ and $Cl(V,g)\otimes T^*V$.
\begin{thm}([DR1,2], [SU]).
The closure of the Bochner integral
\[
\E_B[\o]:=\int_V |\nabla \o(x)|^2\, m_g(dx)\qquad \o\in C^\infty_c(Cl()V,g)
\]
is a completely Dirichlet form on $L^2(Cl(V,g),\tau)$.
\end{thm}
The self-adjoint operator $\Delta_B:=\nabla^*\nabla$ associated to $\E_B$, called the Bochner or connection Laplacian, thus generates a completely Markovian semigroup on $L^2(Cl(V,g),\tau)$ which is strongly Feller in the sense that it reduces to a strongly continuous Markovian semigroup on the Clifford algebra $C^*_0(V,g)$.\\
We may base the proof of the above result on Theorem 5.2 above. Notice first that on the Hilbert space $L^2(Cl(V,g)\otimes T^*V)$ we may consider the C$^*_0(V,g)$-bimodule structure given by
\[
\sigma_1\cdot(\sigma_2\otimes\o)\cdot\sigma_3:=(\sigma_1\cdot\sigma_2\cdot\sigma_3)\otimes\o
\]
and a symmetry $J$ given by
\[
J(\sigma_2\otimes\o):=\sigma_2^*\otimes{\bar\o}
\]
for $\sigma_1,\sigma_3\in C^*_0(V,g)$ and $\sigma_2\otimes\o\in L^2(Cl(V,g)\otimes T^*V)$, where $\sigma_2\to\sigma_2^*$ is the extension of the involution on the Clifford algebra and $\o\to\bar\o$ is the involution on the complexified cotangent bundle. Denoting by $i_X$ the contraction operator with respect to a smooth vector field $X$, we may consider the covariant derivative along $X$ given by $\nabla_X:=i_X\circ\nabla$. Since, by definition, the Levi-Civita connection is a metric connection, we have the identity
\[
\nabla(f\sigma)=\sigma\otimes df+f\nabla\sigma\qquad X(\sigma|\sigma)=(\nabla_X\sigma|\sigma)+(\sigma|\nabla_X\sigma)
\]
for any smooth section $\sigma$ of the Clifford bundle and any smooth function $f$. Since the contraction $i_X$ commutes with actions of the Clifford algebra we have
\[
i_X(\nabla(\sigma\cdot\sigma))=i_X((\nabla\sigma)\cdot\sigma+\sigma\cdot(\nabla\sigma))
\]
for all smooth vector fields $X$ so that
\[
\nabla(\sigma\cdot\sigma)=(\nabla\sigma)\cdot\sigma+(\nabla\sigma)\cdot\sigma,
\]
from which the Leibniz property follows by polarization. Notice that this result is independent upon the Riemannian curvature of the manifold. The situation changes drastically if we consider the Dirac Laplacian $D^2$ on $L^2(Cl(V,g),\tau)$ or better the Dirac quadratic form
\[
\E_D[\sigma]:=\int_Vm_g(dx)\,|D\sigma(x)|^2.
\]
Recall that the Dirac operator $D$ is defined locally for smooth sections $\sigma$ by
\[
(D\sigma)(x):=\sum_{k=1}^n e_k(x)\cdot (\nabla_{e_k}\sigma)(x),
\]
where the vector fields $e_k$ form an orthonormal base in a neighborhood of $x\in V$. At an Hilbert space level, the Dirac operator on the Clifford bundle is isomorphic to the de Rham operator on the exterior bundle
\[
D\simeq d+d^*
\]
and the Dirac Laplacian is isomorphic to the Hodge-de Rham Laplacian
\[
D^2\simeq (d+d^*)^2=dd^*+d^* d.
\]
Differently from the Bochner Laplacian, the potential theoretic properties of $D^2$ depend, however, strongly on the sign of the curvature
\begin{thm} ([CS2])
The following conditions are equivalent
\item i) the Dirac quadratic form $\E_D$ on $L^2(Cl(V,g),\tau)$ is a completely Dirichlet form
\item ii) the curvature operator of $V$ is nonnegative $\widehat R\ge 0$
\item iii) Dirac heat semigroup $e^{-tD^2}$ on the Clifford algebra $C^*_0(V,g)$ is completely Markovian.
\end{thm}
To describe the main steps of the proof let us recall that the curvature tensor $R$ of the metric defines the curvature operator $\widehat R$ on the Hermitian bundle $\Lambda^2 V$ as follows
\[
\widehat R_x:\Lambda^2_x V\to \Lambda^2_x V\qquad (\widehat R_x (v_1\wedge v_2)|v_3\wedge v_4)_{\Lambda^2_x V}:=R_x(v_1,v_2,v_3,v_4)\qquad v_1,v_2,v_3,v_4\in T_xV.
\]
The curvatures identities imply that $\widehat R_x$ is symmetric and thus self-adjoint when extended on the Hilbert space obtained complexifying $\Lambda^2_x V$. By the Bochner identity we have
\[
\E_D=\E_B+Q_R
\]
where $Q_R$ is the quadratic form on $L^2(Cl(V,g),\tau)$ given by
\[
Q_R[\sigma]=\int_V m(dx)\, Q_R(x)[\sigma_x]\qquad Q_R(x)[\sigma_x]=\sum_{\alpha=1}^{n(n-1)/2}\mu_\alpha(x) \|[\eta_\alpha(x),\sigma_x]\|^2
\]
where the norms are those of the Hilbert spaces $L^2(Cl(T_xV,g_x),\tau_x)$, $\eta_\alpha(x)\in \Lambda_x^2 V$ are eigenvectors of $\widehat R_x$ and $\mu_\alpha(x)\in\R$ the corresponding eigenvalues. To prove that i) implies ii) one notices that, by the Albeverio-Hoegh-Khron Theorem 4.5 or by the fact that commutators are bounded derivations, if the curvature operator is nonnegative, then all the eigenvalues are nonnegative and $Q_R$ appears as a superposition of Dirichlet forms. Since, by the Davies-Rothaus Theorem 5.4 above, $\E_B$ is a Dirichlet form, we have that $\E_D$ is a Dirichlet form too. The proof that i) implies ii) the main idea is to use the decomposition theory of derivations to prove that $Q_R$ is a Dirichlet forms because it coincides with the approximately bounded part of the Dirichlet form $\E_D$. Then, a careful analysis of the structure of the Dirichlet forms on the Clifford algebras of finite dimensional Euclidean spaces allows to conclude that all the eigenvalues or the curvature operator are nonnegative. Since $D^2$ is, by construction, a symmetric operator on $L^2(Cl(V,g),\tau)$, if we assume that the Dirac heat semigroup is completely Markovian on the Clifford algebra $C^*_0(V,g)$ then we get that it is a completely Markovian on $L^2(Cl(V,g),\tau)$ so that the quadratic form $\E_D$ is completely Dirichlet.  To prove that $\widehat R\ge 0$ implies that the Dirac semigroup leaves globally invariant the Clifford algebra and that there is strongly continuous one uses i) {\it ellipticity} of $D^2$ to deduce that $e^{-tD^2}$ transforms compactly supported smooth sections of the Clifford bundle into bounded continuous sections ii) {\it Markovianity}, to reduce the problems to the algebra $C_0(V)$ of continuous functions and iii) the fact that $\widehat R\ge 0$ implies that the Ricci curvature is nonnegative so that on $C_0(V)$ the Feller property holds true by a classical result.

\subsection{Voiculescu Dirichlet form in Free Probability}
In this section we describe a Dirichlet form appearing in Free Probability Theory discovered by Dan Virgil Voiculescu [V].
\vskip0.1truecm\noindent
Let $(M,\tau)$ be a noncommutative probability space, i.e. a von Neumann algebra endowed with a faithful, normal trace state. Let us fix a unital $^*$-subalgebra $1_M\in B\subset M$ and a finite set
$X :=\{X_1, . . . , X_n\} \subset M$ of noncommutative random variables, i.e. self-adjoint elements of $M$, {\it algebraically free with respect to} $B$. Let us consider the $^*$-subalgebra $B[X] \subset M$ generated by $X$ and $B$ (regarded as the algebra of {\it noncommutative polynomials in the variables} X with {\it coefficients} in the algebra $B$) and the von Neumann subalgebra $N \subset M$ generated by $B[X]$. Let $HS(L^2(N, \tau)) = L^2(N,\tau)\otimes L^2(N,\tau)$ be the Hilbert $N$-bimodule of Hilbert-Schmidt operators on $L^2(N,\tau)$ and $1_M\otimes 1_M\in HS(L^2(N,\tau))$ the rank one projection onto the multiples of the unit $1_M \in M \subset L^2(M,\tau)$.
\vskip0.1truecm\noindent
Within this framework, D.V. Voiculescu introduced a natural differential calculus with associated Dirichlet form.
\begin{thm}
There exists a unique derivation $\partial_{X_j}:B[X]\to HS(L^2(M,\tau))$ for any fixed $j=1,\cdots, n$ such that
\vskip0.1truecm\noindent
i) $\partial_{X_j}X_k=\delta_{jk}1_M\otimes1_M\qquad k=1,\cdots,n$
\vskip0.1truecm\noindent
ii) $\partial_{X_j}b=0$ for all $b\in B$.
\vskip0.1truecm\noindent
iii) The derivation $(\partial_{X_j}, B[X])$ is densely defined in $L^2(M,\tau)$, symmetric and it is closable if $1_M\otimes 1_M\in D(\partial^*_{X_j})$.
\vskip0.1truecm\noindent
iv) If $1_M\otimes 1_M\in \bigcap_{j=1}^n D(\partial^*_{X_j})$ the quadratic form $(\E,\F)$ defined as
\[
\E[a]:=\sum_{j=1}^n \|\partial_{X_j}a\|^2_{HS(L^2(M,\tau))}\qquad a\in \F:=B[X]
\]
is closable and its closure is completely Dirichlet form.
\end{thm}
Under the assumption $1_M\otimes 1_M\in \bigcap_{j=1}^n D(\partial^*_{X_j})$, the {\it Noncommutative Hilbert Transform of $X$ with respect to $B$} is defined as
\[
\mathcal{I}(X:B):=\Bigl(\sum_{j=1}^n \partial_{X_j}\Bigr)1_M\otimes 1_M\in L^2(M,\tau)
\]
and the {\it Relative Free Fisher information of X with respect to B} is defined as
\[
\Phi^*(X:B):=\|\mathcal{I}(X:B)\|^2_{L^2(M,\tau)}.
\]
It has been shown by P. Biane [Bia1] that this Dirichlet form is the Hessian of the Relative Free Fischer information on the domain where the Relative Free Fisher information is finite. Moreover, if $B=\mathbb{C}$ and still under the assumption that $1_M\otimes 1_M\in \bigcap_{j=1}^n D(\partial^*_{X_j})$, one has the following surprising spectral characterization of {\it semicircular random variables} $X$
\begin{thm} ([Bia1])
A Free Poincar\'e inequality holds true for some $c>0$
\[
c\cdot \|a-\tau (a)\|^2_2\le\E[a]\qquad a\in L^2(M,\tau)
\]
if and only if the random variable $X$ is centered, it has unital covariance and semicircular distribution.
\end{thm}
In the case of semi-circular systems, the self-adjoint operator associated to the Dirichlet form $\E$ is unitarily equivalent to the Number operator on the Free Fock space, it generates the Free Ornstein-Uhlenbeck semi-group and a logarithmic Sobolev inequality holds true (see [Bia2]).

\section{Dirichlet forms on standard forms of von Neumann algebras}
The theory of noncommutative Dirichlet forms illustrated so far has been introduced by S. Albeverio and R. Hoegh-Krohn and developed independently by J.-L. Sauvageot [S1,2,3,4,5,6,7] and by M. Lindsay and E.B. Davies [DL1,2]. We have seen that it can be applied to several fields in which the relevant algebra of observables, to retain a physical language, is no more commutative but it requires, however, that  the reference weight or state is a trace.\\
In this section we describe the extension of the theory to cases in which the reference functional is a normal state on a von Neumann algebra. In a forthcoming section we will describe how this theory may be used to study KMS-symmetric semigroups on C$^*$-algebras as it is required for applications to Quantum Statistical Mechanics, Quantum Field Theory, Quantum Probability and Noncommutative Geometry. Notice that by a fundamental result of G.F. Dell'Antonio [Del], the von Neumann algebras appearing in Quantum Field Theory are typically of type $III$ so that no normal trace is available on them.\\
The exposition is based on the approach given in [Cip1,3] which work on general standard forms of $\sigma$-finite von Neumann algebras. An approach based on the Haagerup standard form [H2] of von Neumann algebras is given in [GL1,2]. This last one has been generalized in [GL3] to consider the reference positive functional on a von Neumann algebra to be a weight. In this respect one ought to consult also [SV Appendix] for the correction of some of the results in [GL3].
\vskip0.2truecm
The Potential Theory developed by A. Beurling and J. Deny and in particular the one to one correspondence between Dirichlet forms and symmetric Markovian semigroups on a measured space $(X,m)$, relies on the geometric properties of the cone $L^2_+(X,m)$ of positive functions in the Hilbert space $L^2(X,m)$. This  is a {\it closed, convex cone}  which is {\it self-polar} in the sense that
\[
a\in L^2_+(X,m)\quad{\rm if\,\,and\,\,only\,\,if}\quad (a|b)\ge 0\quad {\rm for\,\,all}\quad b\in L^2_+(X,m).
\]
The theory of noncommutative Dirichlet forms developed by S. Albeverio and R. Hoegh-Khron on C$^*$-algebras endowed with a faithful, semifinite trace $(A,\tau)$ is based on analogous properties of the cone $L^2_+(A,\tau)$ defined as the closure in the GNS Hilbert space $L^2(A,\tau)$ of the cone $\{a\in A_+:\tau(a)<+\infty\}$. This cone determines, in particular, an anti-unitary involution $J_\tau$ on $L^2(A,\tau)$ which extends the isometric involution $a\mapsto a^*$ of $A$ to the von Neumann algebra $L^\infty(A,\tau)$. The whole structure $(L^\infty(A,\tau), L^2(A,\tau), L^2_+(A,\tau), J_\tau)$ realizes {\it the standard form} of the von Neumann algebra $L^\infty(A,\tau)$ in the following sense.
\begin{defn}(Standard form of a von Neumann algebra) ([Ara], [Co1], [H1]).\\
A {\it standard form} $(\M ,\H , \P, J)$ of a von Neumann algebra $\M$ acting on a Hilbert space $\H$, consists of a {\it self-polar cone} $\P\subset\H$ and an anti-unitary involution $J$, satisfying:
\vskip0.2truecm\noindent
i) $J\M J=\M^\prime$;
\vskip0.2truecm\noindent
ii) $JxJ = x^*\, ,\qquad \forall x\in \M\cap \M^\prime\quad$ (the center of $\M$);
\vskip0.2truecm\noindent
iii) $J\eta = \eta\, ,\qquad \forall\eta\in \P$;
\vskip0.2truecm\noindent
iv) $xJxJ (\P )\subseteq \P\, ,\qquad\forall x\in \M\, .$
\end{defn}\noindent
The $J$-{\it real part} of $\H$ is defined as $\H^J:=\{\xi\in\H:J\xi=\xi\}$ and one has the decomposition $\H=\H^J\oplus i\H^J$. Moreover, one may define the {\it positive part} $\xi_+\in\P$ of $J$-real vector $\xi\in\H^J$ as the Hilbert projection of $\xi$ onto the positive cone $\P$, its {\it negative part} $\xi_-\in\P$ as the difference $\xi_-:=\xi-\xi_+$ and its modulus by $|\xi|:=\xi_++\xi_-\in\P$ so that $\xi=\xi_+-\xi_-$ and $(\xi_+|\xi_-)=0$.

\subsubsection{Standard form of commutative von Neumann algebras}
One may readily checks that $(L^\infty (X,m), L^2 (X,m), L^2_+ (X,m), J )$ is a standard form of the commutative von Neumann algebra $L^\infty (X,m)$ (once the anti-unitary involution is given by the complex conjugation: $Ja=\overline{a}$) and that the above notions related to the order structure assume the familiar  meaning.

\subsubsection{Hilbert-Schmidt standard form}
A noncommutative example is provided by Hilbert-Schmidt standard form
\[
(B (h)\, , L^2 (h)\, , L^2_+ (h)\, ,J)
\]
of the algebra $B (h)$ of all bounded operators on a Hilbert space $h$. In the Hilbert space $L^2 (h)$ of all Hilbert-Schmidt operators on $h$, the cone $L^2_+ (h)$ of the positive ones is self-polar and the involution $J$ associates to the Hilbert-Schmidt operator $a$ its adjoint $a^*$.
\vskip0.2truecm\noindent
Essential properties of the standard form of a von Neumann algebra are its {\it existence and uniqueness} (modulo unitaries preserving the positive cones). These properties authorize to denote the standard form of a von Neumann algebra $M$ simply by
\[
(M,L^2(M),L^2_+(M),J).
\]
These main results are also of pratical use because different standard forms may show different advantages (or inconveniences). In the commutative case uniqueness is a reflection of the fact that the von Neumann algebra $L^\infty (X,m)$ is determined by the class of zero $m$-measure sets only, so that the algebra can be represented standardly on the space $L^2(X,m')$ of any measure $m'$ equivalent to $m$.

\subsubsection{Standard form of semifinite von Neumann algebras}
In case the von Neumann algebras $\M$ is semifinite, a standard form may be constructed by the GNS representation associated to a normal, semifinite trace $\tau$ on $M$ as
\[
(M, L^2(M,\tau), L^2_+(M,\tau), J_\tau).
\]

To construct the standard form of a von Neumann algebra $M$ starting from a normal state $\o_0\in M_{*+}$ one need to recall some aspects of the Tomita-Takesaki Modular Theory of von Neumann algebras [T1,2]. We may assume that $M$ is represented in a Hilbert space $\H$ so that $M\subseteq B(\H)$ and $\o_0$ is represented by a cyclic and separating vector $\xi_0\in h$ as $\o_0(x)=(\xi_0|x\xi_0)_\H$ for $x\in M$ (for example, $\H$ can be assumed to be the GNS space $L^2(M,\o_0)$). The anti-linear map $S(x\xi_0):=x^*\xi_0$, densely defined on
$M\xi_0\subseteq \H$, is a closable operator on $\H$ and we may consider the polar decomposition of its closure $\bar S$
\[
\bar S=J\Delta_0^{1/2}
\]
where the square root of the self-.adjoint {\it modular operator} $\Delta_0=\bar{S}^*\bar{S}$ provides its positive part and the {\it modular conjugation} $J$ is an anti-unitary operator on $\H$ providing the phase. Using these tools one proves that
\[
\P:=\overline{\{xJxJ\in \H:x\in M\}}
\]
is a self-polar cone in $\H$ coinciding with $\overline{\Delta_0^{1/4}M_+\xi_0}$ and that $(M,\H,\P,J)$ is a standard form. When $\o_0$ is a trace state then $S$ is isometric so that the modular operator $\Delta_0$ reduces to the identity, $S=J$ and $\P=\overline{M_+\xi_0}$. The modular operator $\Delta_0$ measures how much the state $\o_0$ differs from a trace state in that only in this case $\Delta_0$ reduces to the identity.
\vskip0.2truecm\noindent
The denomination {\it modular} used for the operator $\Delta_0$ originates from the following example.

\subsubsection{Modular operator and standard form of group von Neumann algebra}
Let $(G,m_H)$ be a locally compact group and consider the convolution algebra $C_c(G)$ acting by left convolution $\lambda_G$ on $L^2(G,m_H)$ and define the {\it group von Neumann algebra} $L(G)$ as $\lambda_G(C_c(G))''$. The Haar measure $m_H$ determines an additive, homogeneous, lower semicontinuous functional $\o_H$ on the positive part $L(G)_+$, called the {\it Plancherel weight} (see [T3 Chapter VII]).
%, by
%\[
%\o_H(\lambda_G(a)):=\int_G a(s)\, m_H(ds)\qquad a\in C_c(G).
%\]
It is a trace if and only if $G$ is unimodular and a trace state if and only if $G$ is discrete. Since on $L(G)$ the involution is determined by $\lambda_G(a)^*:=\lambda_G(a^*)$ where $a^*(s):=\overline{a(s^{-1})}$ for $s\in G$ and $a\in C_c(G)$, one may check that the modular operator $\Delta_H$  on $L^2(G,m_H)$ associated to the Plancherel weight $\o_H$ is given by the multiplication operator by the {\it modular function} $G\ni s\mapsto dm_H(\cdot s^{-1})/dm_H$.

\subsubsection{Modular operators and Gibbs states}
On the von Neumann algebra $B(h)$ any normal state $\o_0$ can be represented by a self-adjoint, positive, compact operator $\rho\in B(h)$ having unit trace, called {\it density matrix}, as follows
\[
\o_0(x)={\rm Tr\,}(x\rho)\qquad x\in B(h).
\]
setting $H:=-\ln\rho$ we have $\rho=e^{-H}$ so that $\o_\rho$ appears as the Gibbs equilibrium state of the dynamical system whose time evolution is given by the automorphisms group
\[
\alpha_t(x)=e^{-itH}xe^{+itH}\qquad x\in B(h)
\]
generated by the Hamiltonian $H$.
We now use the Hilbert-Schmidt standard form of $B(h)$ to compute the action of the modular operator. Since $\o_0(x)={\rm Tr\,}(x\rho)={\rm Tr\,}(\rho^{1/2}x\rho^{1/2})=(\rho^{1/2}|x\rho^{1/2})_{L^2(h)}$ we have that compact operator $\xi_0=\rho^{1/2}\in L^2(h)$ in the Hilbert-Schmidt class is the cyclic and separating vector representing $\o_0$. To recover the action of the modular operator notice that, by definition, we have $J\Delta^{1/2}_0 (x\xi_\rho ) = x^*\xi_\rho$ for $x\in\B (h)$. Then $\Delta^{1/2}_0 (x\rho^{1/2})=J(x^*\rho^{1/2})=\rho^{1/2}x$ for all $x\in\B (h)$ so that
\[
\Delta^{1/2}_0 \xi = \rho^{1/2}\xi\rho^{-1/2}\qquad \xi\in D(\Delta^{1/2}_0 ):=\{\eta\in L^2 (h) : \rho^{1/2}\eta\rho^{-1/2}\in L^2 (h)\}.
\]
Notice that $(xJxJ)(\xi)=x\xi x^*$ for all $x\in B(h)$ and $\xi\in L^2(h)$ so that $\P=L^2_+(h)$.

Another crucial property of the standard form is that any normal state $\o\in\M_{*+}$ can be represented as the vector state of a unique, unit vector $\xi_\o\in\P$ in the standard positive cone, i.e. $\o(x)=(\xi_\o|x\xi_\o)_\H$ for all $x\in\M$.

\subsection{Tomita-Takesaki Theory and Connes' Radon-Nikodym Theorem}([T1,2], [Co2]).
Let $(M,\o)$ be a von Neumann algebra with a faithful, normal state and denote by $\pi_\o:M\to B(L^2(M,\o))$ the associated GNS representation. The Tomita-Takesaki Theorem then ensure that
\[
J_\o \pi_\o(M) J_\o=\pi_\o(M)',
\]
\[
\Delta_\o^{-it}\pi_\o(M)\Delta_\o^{it}=\pi_\o(M)\qquad t\in\R.
\]
Moreover, setting
\[
\sigma^\o_t:M\to M\qquad \sigma^\o_t (x):=\pi_\o^{-1}(\Delta_\o^{-it}\pi_\o(x)\Delta_\o^{it})\qquad x\in M,\qquad t\in\R
\]
one gets a w$^*$-continuous group $\sigma^\o\in {\rm Aut}(M)$ of automorphisms of the von Neumann algebra that {\it satisfies and  it is uniquely determined} by the following {\it modular condition}
\[
\o(x\sigma^\o_{-i}(y))=\o(yx)
\]
for all $x,y\in M$ which are analytic with respect to $\sigma^\o$. A fundamental theorem due to A. Connes [Co0 Theorem 1.2.1], which has to be considered as the noncommutative generalization of the Radon-Nikodym Theorem, states that the modular automorphism group of a von Neumann algebra is essentially unique: for any pair $\phi,\psi\in M_{*+}$ of faithful, normal states on $M$, there exists a canonical 1-cocycle $u:\R\to\mathcal{U}(M)$ for $\sigma^\phi_t$,with values in the unitary group of $M$
\[
u_{t_1+t_2}=u_{t_1}\sigma^\phi_{t_1}(u_{t_2})\qquad t_1,t_2\in\R,
\]
such that
\[
\sigma^\psi_t(x)=u_t\sigma^\phi_t(x)u_t^*\qquad x\in M,\quad t\in\R.
\]
\subsubsection{Modular operators on type I factors}
In case of the von Neumann algebra $\B (h)$ and the normal state $\o (x):= {\rm Tr} (\rho x)$ associated to a positive, trace-class operator $\rho\in \B(h)$ with unit trace, one checks that the modular group is given by
\[
\sigma_t^\o (x)=\rho^{it}x\rho^{-it}\, ,\qquad x\in \B (h)\, ,\qquad t\in\mathbb{R}
\]
and that the modular condition follows from the trace property of $\rm Tr$
\[
\o(yx)={\rm Tr}(\rho yx)={\rm Tr}(\rho x \rho y\rho^{-1})=\o(x\sigma^\o_{-i}(y)).
\]
In the particular case of a matrix algebra $M_n(\mathbb{C})$, denoting by $e_{jk}$ the matrix units, if the density matrix $\rho$ is diagonal with eigenvalues $\lambda_1,\cdots,\lambda_n>0$, one has
\[
\sigma^\o_t(e_{jk})=\Bigl(\frac{\lambda_j}{\lambda_k}\Bigr)^{it}e_{jk}\qquad j,k=1,\cdots,n,\quad t\in\R.
\]

\subsection{Symmetric embeddings}([Ara], [Cip1]).
In the commutative case, the standard form of a probability space $(X,m)$, we have the natural embeddings $L^\infty(X,m)\subseteq L^2(X,m)$, $L^2(X,m)\subseteq L^1(X,m)$ and $L^\infty(X,m)\subseteq L^1(X,m)$.\\
These may be generalized to the standard form of any von Neumann algebra $M$, using the modular operators associated to any fixed faithful normal state $\o\in M_{*+}$.
\begin{defn} (Symmetric embeddings) The {\it symmetric embeddings} associated to the standard form $(M ,\H , \P , J)$ and a cyclic and separating vector $\xi_\o\in \P$ are defined as follows:
\vskip0.2truecm\noindent
i) $i_\o :M\to\H\qquad i_\o (x):= \Delta_\o^{1/4}x\xi_\o\, ,\qquad x\in M$;
\vskip0.2truecm\noindent
ii) $i_{\o*} :\H\to M_*\qquad \langle i_{\o*}(\xi), y\rangle = (i_\o (y^*) |\, \xi )=(\Delta_\o^{1/4} y^* \xi_\o |\, \xi )\, ,\qquad \xi\in\H\, ,\quad y\in  M$;
\vskip0.2truecm\noindent
iii) $j_\o :M\to M_*\qquad \langle j_\o (x) , y\rangle=  (J_\o y\xi_\o |x\xi_\o)\, ,\qquad x,y\in M$.
\vskip0.2truecm\noindent
These maps are well defined because $\M\xi_\o\subseteq D(\Delta_\o^{1/2})$, by the very definition of the modular operator, and because $D(\Delta_\o^{1/2})\subseteq D(\Delta_\o^{1/4})$ by the Spectral Theorem.
\end{defn}\noindent
The essential feature of these embeddings is that they preserve the order structures of $\M$, $\H$ and $M_*$ provided by the positive cones of these spaces. In particular $i_\o$ establishes a one to one homeomorphism between the set $[0,1_M]:=\{x=x^*\in M:0\le x\le 1_M\}\subset M_+$ and its image $i_\o([0,1_M])=\{\xi\in\P:0\le \xi\le\xi_\o\}:=[0,\xi_\o]\subset\P$.
\vskip0.2truecm\noindent
In the following we will indicate by $\xi\wedge \xi_0$ the Hilbert projection onto the closed, convex set $\{\xi\in\H^J:\xi\le\xi_\o\}$ of a $J$-real vector $\xi\in\H^J$. In the commutative case and when $\xi_\o$ is given by the constant function $1$ and $a\in L^2_\R(X,m)$ is a real function, then $a\wedge 1$ reduces to the so called {\it unit contraction} of $a$ given by $(a\wedge 1)(x)=\inf (a(x),1)$ for $x\in X$. Using this geometric operation we may rephrase on the standard form of any von Neumann algebra $M$, the Markovianity of Dirichlet forms one considers in the commutative setting $L^\infty(X,m)$.

\begin{defn}(Dirichlet forms [Cip1]).
Let $(M,\H,\P,J)$ be a standard form of a von Neumann algebra $M$ and $\xi_\o\in \P$ a cyclic and separating vector representing the normal state $\o\in M_{*+}$.\par\noindent
A quadratic form $\E : \H\to (-\infty ,+\infty]$ is said to be {\it $J$-real} if
\[
\E[J\xi]=\E[\xi]\qquad \xi\in\H
\]
and {\it Markovian} with respect to $\xi_\o$ if it is $J$-real and
\begin{equation}
\E[\xi\wedge \xi_\o]\le\E[\xi]\qquad \forall\, \xi\in \H^J\, .
\end{equation}
In case $\E[\xi_\o]=0$, the Markovianity condition is equivalent to
\[
\E(\xi_+|\xi_-)\le 0\qquad \xi=J\xi\in\H.
\]
A {\it densely defined, lower semicontinuous Markovian form is called a Dirichlet form} with respect to $\xi_\o$ or $\o$.\\
The quadratic form $\E$ is called a {\it completely Dirichlet form} if any of its matrix extension  $\E_n$ on $\H\otimes L^2(\mathbb{M}_n(\mathbb{C})$, given by
\[
\E_n[[a_{i,j}]_{i,j=1}^n]=\sum_{i,j=1}^n\E[a_{i,j}]\qquad [a_{i,j}]_{i,j=1}^n\in\H\otimes L^2(\mathbb{M}_n(\mathbb{C}),
\]
is a Dirichlet form w.r.t. the state $\o_0\otimes\tau_n$, where $\tau_n$ is the unique tracial state on $\mathbb{M}_n(\mathbb{C})$.
\end{defn}
In particular, if $\E$ is Markovian and $\E[\xi]$ is finite then $\E[\xi\wedge\xi_\o]$ is finite too. Also, in the commutative setting Dirichlet forms are automatically completely Dirichlet forms.
In other words, under the Hilbertian projection $\xi\mapsto \xi\wedge\xi_\o$, the value of the quadratic form does not increase. As noticed above, this definition reduces to the usual one in the commutative setting. We are going to see that in any standard form, Dirichlet forms represent an infinitesimal characterization of strongly continuous, symmetric Markovian semigroups.
\begin{thm} (Characterization of Markovian semigroups by Dirichlet forms [GL1,2], [Cip1].)\\
 Let $(M \, ,\H\, , \P \, J)$ be a standard form of a von Neumann algebra $M$ and $\xi_\o\in \P$
the cyclic vector representing a state $\o\in M_{*+}$. Let $\{T_t : t\ge 0\}$ be a $J$-real, symmetric, strongly continuous, semigroup on the Hilbert space $\H$ and $\E : \H\to (-\infty,+\infty]$ the associated $J$-real, lower semibounded, closed  quadratic form. Then, the following properties are equivalent
\vskip0.2truecm\noindent
i) $\{T_t : t\ge 0\}$ is Markovian with respect to $\xi_\o$;
\vskip0.2truecm\noindent
ii) $\E$ is a Dirichlet form with respect to $\xi_\o$.
\vskip0.2truecm\noindent
In particular, Dirichlet forms are automatically nonnegative and Markovian semigroups are automatically contractive.
\end{thm}\noindent
Dirichlet forms not only determine and are determined by strongly continuous Markovian semigroups on the standard Hilbert space, but they are also in one-to-one correspondence with point-weak*-continuous, completely positive, subunital (abbreviated with Markovian) semigroups on the von Neumann algebra, satisfying a certain {\it modular symmetry property} which is a deformation of the modular condition.

\begin{thm} (Markovian semigroups on standard forms of von Neumann algebras [Cip1]).
Let $(\M \, ,\H\, , \P \, J)$ be a standard form of a von Neumann algebra $\M$, $\o\in M_{*+}$ a faithful state and $\xi_\o\in \P$ its representing cyclic vector. Then there exists a one-to-one correspondence between
\vskip0.1truecm\noindent
i) Markovian (with respect to $\xi_\o$) semigroups $\{T_t:t>0\}$ on $L^2(M)$ and
\vskip0.2truecm\noindent
ii) Markovian semigroups $\{S_t:t>0\}$ on $M$ which are {$\o$-modular symmetric} in the sense that, for all $x,y$ in a weak*-dense, ${\bf \sigma^\o}$-invariant ,$^*$-subalgebra of $M_{\sigma^\o}$, one has
\begin{equation}
\o \Bigl(y S_t (x)\Bigr)=\o \Bigl(\sigma^\o_{\frac{i}{2}}(x)S_t (\sigma^\o_{-\frac{i}{2}}(y))\Bigr).
\end{equation}
The correspondence is provided by the symmetric embedding through the relation
\[
i_\o:M\to L^2(M)\qquad i_\o\circ S_t=T_t\circ i_\o.
\]
\end{thm}
\begin{rem}
A careful analysis of the family of closed cones $\overline{\{\Delta_\o^\alpha x\xi_\o:x\in M_+\}}$ in $L^2(M)$ for $\alpha\in (0,1/2)$, indicates that if instead of the symmetric embedding, the so called {\it GNS embedding} $x\mapsto x\xi_\o$ of $M$ into $L^2(M)$ is used, the resulting semigroup automatically commutes with modular operator. In application to convergence to equilibrium in Quantum Statistical Mechanics this situation should be avoided and this is the reason why the self-polar cone corresponding to $\alpha=1/4$ is used.
\end{rem}
\subsubsection{Elementary Dirichlet forms}
As a first example of a Dirichlet form with respect to a not necessarily trace state, we illustrate a construction that can be considered a generalization of the one of Albeverio-Hoegh-Khron in Theorem 4.5 above. Elementary Dirichlet form will find application to approximation/rigidity properties of von Neumann algebras in Section 7.6.
\vskip0.2truecm\noindent
Let $(\M \, ,\H\, , \P \, J)$ be a standard form and $\xi_0\in \P$ a cyclic vector. Consider, for fixed $a_k\in\M$, $\mu_k\, ,\nu_k >0$ and $k=1,\dots ,n$, the operators
\[
\partial_k :\H\to\H\qquad d_k := i(\mu_k a_k -\nu_k j(a_k^*))
\]
and the bounded quadratic form
\[
\E[\xi]:=\sum_{k=1}^n\|\partial_k\xi\|^2\qquad \xi\in\H.
\]
Then $\E$ is $J$-real and $\E(\xi_+|\xi_-)\le 0$ for all $J$-real $\xi\in\H$ if and only if
\[
\sum_{k=1}^n [\mu_k^2 a_k^* a_k -\nu_k^2 a_k a_k^*]\in \M\cap\M^\prime.
\]
Moreover, if the above condition holds true, $\E$ is a Dirichlet form with $\E[\xi_0]=0$ if and only if the numbers $(\mu_k /\nu_k)^2$, $k=1,\dots ,n$, are eigenvalues of the modular operator $\Delta_{\xi_0}$, corresponding the eigenvectors $a_k\xi_0\in\H$. Conditions like the one above are considered in the framework of q-deformed CCR relations and related factor von Neumann algebras [Boz].\\
The construction above provides a, possibly unbounded, Dirichlet form even when $n=\infty$, provided $\E$ is densely defined.

\subsubsection{Quantum Ornstein-Uhlenbeck and Quantum Brownian motion semigroups} ([CFL]). We describe here the construction of a Dirichlet form, on the Neumann algebra $B(h)$, which generates a Markovian semigroups appearing in quantum optics.
\vskip0.2truecm\noindent
On the Hilbert space $h:= l^2 (\mathbb{N})$ consider the standard form $(B (h), L^2 (h), L^2_+ (h), J)$. Let $\{e_n :n\ge 0\}\subset h$ be the canonical Hilbert basis, and denote by $|\, e_m\rangle\langle e_n |$, $n,m\ge 0$, the partial isometries, having $\mathbb{C}e_n$ as initial space and $\mathbb{C}e_m$ as final one.
\vskip0.2truecm\noindent
Fix the parameters $\mu >\lambda >0$, set $\nu :=\lambda^2 /\mu^2$ and let $\o_\nu (x):={\rm Tr} (\rho_\nu x)$ the normal state on $B (h)$ represented by the density matrix
\[
\rho_\nu :=(1-\nu)\sum_{n\ge 0}\nu^n |e_n\rangle\langle e_n |\, .
\]
The state $\o_\nu$ is then represented by the cyclic vector $\xi_\nu :=\rho_\nu^{1/2}=(1-\nu)^{1/2}\sum_{n\ge 0}\nu^{n/2} |e_n\rangle\langle e_n |$. The {\it creation} and {\it annihilation operators}, $a^*$ and $a$ on $h$, are defined by
\[
a^* e_n :=\sqrt{n+1}e_{n+1}\, ,\qquad ae_n :=\left\{
                                              \begin{array}{ll}
                                                \sqrt{n}e_{n-1}, & \hbox{if $n> 0$;} \\
                                                0, & \hbox{if $n=0$.}
                                              \end{array}
                                            \right.
\]
They are adjoint of one another on their common domain $D(a)=D(a^*)=\bigl\{e\in h:\sum_{n\ge 0}\sqrt{n}|\langle e|\, e_n\rangle|^2 <\infty\bigr\}$
and satisfy the CCR
\[
a\, a^* -a^* a = I\, .
\]
The quadratic form given by
\[
\E[\xi]:=\|\mu a\xi -\lambda\xi a\|^2+\|\mu a\xi^* -\lambda\xi^* a\|^2,
\]
densely defined in $L^2 (h)$ on the subspace of finite rank operators
\[
D(\E):={\rm linear\,\, span}\{|\, e_m\rangle\langle e_n |\, ,n,m\ge 0\},
\]
is closable and Markovian so that its closure is a Dirichlet form with respect to $\o_\nu$, generating the so called {\it quantum Ornstein-Uhlenbeck} Markovian semigroup. Moreover, since, as it is easy to check one has $\E[\xi_\nu]=0$, it results that the cyclic vector is left invariant by the semigroup.
\vskip0.2truecm\noindent
When $\lambda = \mu$, the role of the invariant state $\o_\nu$ has to be played by the normal, semifinite trace $\tau$ on $B (h)$. However, even in this case, using the Albeverio-Hoegh-Khron criterion, it is possible to prove that the closure of the unbounded quadratic form
\[
\E[\xi]:=\|a\xi -\xi a\|^2+\| a\xi^* -\xi^* a\|^2\, ,\qquad\xi\in D(\E)
\]
is a Dirichlet form. The associated $\tau$-symmetric Markovian semigroup on $B (h)$, may be dilated by a Quantum Stochastic Process, known as the {\it Quantum Brownian motion}. This represents a sort of {\it bridge} between pairs of classical stochastic processes of quite different type. In fact on a suitable, invariant, maximal abelian subalgebra (masa), this semigroup reduces to the semigroup of a classical Brownian motion while on another masa, it reduces to the semigroup of a classical birth and death process.

\section{Application to approximation/rigidity properties\\ of von Neumann algebras}
In this section we describe three results showing that the spectral properties of Dirichlet forms are naturally and deeply connected with those fundamental properties of von Neumann algebras having to do with the ideas of approximation and rigidity.

\subsection{Amenable groups}In 1929 J. von Neumann discovered a far reaching explanation of the Banach-Tarski paradox in terms of a property, called {\it amenability}, of a group of Euclidean motions in $\R^n$ which holds true in dimension $n=1,2$ but it does not in higher dimensions.
\begin{defn}
A discrete group $\Gamma$ is {\it amenable} if there exist a left-translation invariant probability measure on $\Gamma$.
\end{defn}\noindent
This property is equivalent to the existence of a sequence of finitely supported, positive definite functions $\phi_n$ on $\Gamma$, pointwise converging to the constant function $1$,
\[
\lim_n \phi_n(t)=1\qquad {\rm for\,\, all\,\,} t\in\Gamma,
\]
and to the existence of a {\it proper}, conditionally negative definite function $\psi:\Gamma\to\mathbb{C}$ (see 5.1.1). Recall that a function $\phi:\Gamma\to\mathbb{C}$ is positive definite if the matrices $[\phi(s_j^{-1}s_k)]_{j,k=1}^n$ are positive definite for all $s_1\,\cdots, s_k\in\Gamma$, i.e. if for all  $c_1,\cdots c_n\in\mathbb{C}$ one has
\[
 \sum_{j,k=1}^n {\bar c_j}\phi(s_j^{-1}s_k)c_k\ge 0.
\]
%and that a function $\psi:\Gamma\to\mathbb{C}$ is negative definite if for all $s_1\,\cdots, s_k\in\Gamma$ and $c_1,\cdots c_n\in\mathbb{C}$
%\[
%\sum_{k=1}^n c_k=0       \implies       \quad \sum_{j,k=1}^n {\bar c_j}\psi(s_j^{-1}s_k)c_k\le 0.
%\]
Since positive definite functions are coefficients of unitary representations and the constant function $1$ is the coefficient of the trivial representation, amenability is also equivalent to the fact that the trivial representation is weakly contained in the left regular one, i.e. it is unitarily equivalent to a subrepresentation of a multiple of the regular representation.
\vskip0.2truecm\noindent
The amenability of a group $\Gamma$ can be translated in terms of a corresponding property of its associated group von Neumann algebra $L(\Gamma)$.
\vskip0.2truecm\noindent
To introduce this property in complete generality, we recall the fundamental notions of bimodule and correspondence.
\subsection{Bimodules and Connes correspondences}([Po], [Co2], [AP]).
A Banach $M$-{\it bimodule} $E$  on a C$^*$-algebra $M$ is a Banach space $E$ together with a pair of norm continuous, commuting actions of $M$.
\vskip0.2truecm\noindent
If the left action of $x\in M$ on $\xi\in E$ is denoted by $x\xi$ and the right action of $y\in M$ on $\xi\in E$ is denoted by $\xi y$, the required commutation reads $(x\xi)y=x(\xi y)$.
\vskip0.2truecm\noindent
In case $M$ is a von Neumann algebra and $E$ is a {\it dual} bimodule, in the sense that it is the dual Banach space of a predual one, the left and right actions are required to be continuous with respect to the weak*-topology of $E$.
\vskip0.2truecm\noindent
A {\it Connes correspondence} on $M$ is a Hilbert space $\H$ which is an $M$-bimodule.
\vskip0.1truecm\noindent
Denote by $M^\circ$ the {\it opposite algebra} of $M$: it coincides with $M$ as a vector space but the product is taken in the reverse order $x^\circ y^\circ :=(yx)^\circ$ for $x^\circ , y^\circ\in M^\circ$. By convention, elements $y\in M$, when regarded as elements of the opposite algebra are denoted by $y^\circ\in M^\circ$. Let $M\otimes_{\rm max}M^\circ$ the maximal tensor product of $M$ and $M^\circ$ considered as C$^*$-algebras.\\
A correspondence on $M$ is nothing but a representation
\[\pi:M\otimes_{\rm max}M^\circ\to B(\H)\qquad \pi(x\otimes y^\circ)\xi=x\xi y
\]
such that $M\ni x\mapsto \pi(x\otimes 1_M)$ and $M\ni x\mapsto \pi(1_M\otimes x^\circ)$ provide normal representations. Correspondences of von Neumann algebras may be conveniently thought of both as generalization of unitary representations of groups.
\vskip0.2truecm\noindent
Among the correspondences of von Neumann algebras, the following are of central importance.
\vskip0.1truecm\noindent
The {\it identity} or {\it standard} correspondence of a von Neumann algebra $M$ is provided by its standard representation $(M,L^2(M), L^2_+(M),J)$. Here beside the left action of $M$ on $L^2(M)$, denoted by $x\xi$ for $x\in M$ and $\xi\in L^2(M)$, we have the right  action defined by $\xi x:=Jx^*J\xi$.\\
The {\it coarse correspondence} is the $M$-bimodule given by $L^2(M)\otimes \overline{L^2(M)}$ with actions
\[
x(\xi\otimes\overline{\eta})y=x\xi\otimes \overline{\eta y}\qquad x,y\in M,\quad \xi,\eta\in L^2(M).
\]
This is also called the {\it Hilbert-Schmidt correspondence} by the identification of $L^2(M)\otimes \overline{L^2(M)}$ with the Hilbert space $HS(L^2(M))$ of Hilbert-Schmidt operator on $L^2(M)$. In this terms the actions are given by $xTy\in HS(L^2(M))$ for $x,y\in M$ and $T\in HS(L^2(M))$.
\vskip0.1truecm\noindent
Correspondences of von Neumann algebras may also be fruitfully thought as generalization of completely positive maps. In fact, suppose that on $M$ a faithful, normal state $\o$ is fixed and consider a not necessarily self-adjoint, completely Markovian map $T:L^2(M,\o)\to L^2(M,\o)$,  assuming, to simplify, that $T\xi_\o=\xi_\o$. Then the functional determined by
\[
\Phi_T:M\otimes_{\rm max}M^\circ\to \mathbb{C}\qquad \Phi_T(x\otimes y^\circ):=(i_\o(y^*)|Ti_\o(x))
\]
is a state on $M\otimes_{\rm max}M^\circ$ which, by the GNS construction, give rise to a representation of $M\otimes_{\rm max}M^\circ$, thus to a correspondence $\H_T$ on $M$. The unit, cyclic vector $\xi_T\in\H_T$ representing $\Phi_t$ thus satisfies
\[
(i_\o(y^*)|Ti_\o(x))=(\xi_T|x\xi_Ty)_{\H_T}\qquad x,y\in M.
\]
A fundamental operation that is defined on correspondences is their {\it relative tensor product}, by which any $M$-$N$-correspondence $\H_N$ and any $N$-$P$-correspondence $\K_P$ may tensorized, in this order, to produce an $M$-$P$-correspondence denoted by $\H\otimes\K_P$.
The advantages to translate into the common language of correspondences problems of apparently different origin concerning von Neumann algebras, are the possibility to let them play into a shared ground on one side, and the possibility to use the tools of representation theory, for example to introduce notions like containment, weak containment and convergence.
\subsection{Amenable von Neumann algebras}

\begin{defn} ([Co2,3,4], [CE]).
A C$^*$ or von Neumann algebra $M$ is said to be {\it amenable} if for every dual Banach $M$-bimodule $E$, all derivations $\delta:M\to X$, that is maps satisfying the Leibniz property
\[
\delta(ab)=(\delta a)b+a(\delta b)\qquad a,b\in M,
\]
are inner, i.e. there exists $\xi\in E$ such that
\[
\delta(x)=x\xi-\xi x\qquad x\in M.
\]
\end{defn}\noindent
This property was introduced by Johnson and Ringrose in their works on cohomology of operator algebras. As the result of an enormous amount of efforts, it has been shown that amenability is equivalent to approximation properties:
\vskip0.2truecm\noindent
i) a C$^*$-algebra $A$ is amenable  if and only if it is {\it nuclear} in the sense that its identity map can be approximated in the point-norm topology,
\[
\lim_n \|\psi_n\circ\phi_n (a)-a\|=0\qquad {\rm for\,\, all\,\,} a\in A,
\]
by the composition of suitable contractive, completely positive maps
\[
\psi_n:A\to \mathbb{M}_{k_n}(\mathbb{C})\qquad \phi_n:\mathbb{M}_{k_n}(\mathbb{C})\to A;
\]
ii) a von Neumann algebra $M$ is {\it weakly nuclear} if and only if its identity map can be approximated in the point-ultraweak topology,
\[
\lim_n \eta (\psi_n\circ\phi_n (a)-a)=0\qquad {\rm for\,\, all\,\,} a\in A,\quad \eta\in M_*,
\]
by the composition of suitable contractive, completely positive maps
\[
\psi_n:A\to \mathbb{M}_{k_n}(\mathbb{C})\qquad \phi_n:\mathbb{M}_{k_n}(\mathbb{C})\to A;
\]
iii) a von Neumann algebra $M$ is amenable if and only if it is {\it hyperfinite} in the sense that it is generated by an increasing sequence of finite-dimensional subalgebras.
\vskip0.2truecm\noindent
Among the examples of amenable von Neumann algebras, one may recall  i) the group von Neumann algebra of a locally compact amenable group ii) the crossed product of an abelian von Neumann algebra by an amenable locally compact group iii) the commutant von Neumann algebra of any continuous unitary representation of a connected locally compact group and iv) the von Neumann algebra generated by any representation of a nuclear C$^*$-algebra.
\subsection{Amenability and subexponential spectral growth rate of Dirichlet forms} ([CS5]).
To illustrate a first connection between approximation properties of von Neumann algebras and spectral properties of Dirichlet form, we first recall a definition.
\begin{defn}(Spectral growth rate of Dirichlet forms [CS5]).
Let $(N,\o)$ be an infinite dimensional, $\sigma$-finite, von Neumann algebra with a fixed faithful, normal state on it.
\par\noindent
Let $(\E,\F)$ be a Dirichlet form on $L^2(N,\o)$ and let $(L,D(L))$ be the associated nonnegative, self-adjoint operator. Assume that its spectrum  $\sigma(L)=\{\lambda_k\ge 0: k\in\mathbb{N}\}$ is {\it discrete}, set
\[
\Lambda_n :=\{k\in\mathbb{N}: \lambda_k\in [0,n]\}\, ,\qquad \beta_n:=\sharp (\Lambda_n)\, ,\qquad n\in\mathbb{N}
\]
and define the {\it spectral growth rate} of $(\E,\F)$ as
\[
\Omega (\E,\F):=\limsup_{n\in\mathbb{N}} \sqrt[n]{\beta_n}\, .
\]
The Dirichlet form $(\E,\F)$ is said to have
\vskip0.2truecm\noindent
{\it exponential growth} if $(\E,\F)$ has discrete spectrum and $\Omega (\E,\F)>1$
\vskip0.1truecm\noindent
{\it subexponential growth} if $(\E,\F)$ has discrete spectrum and $\Omega (\E,\F)= 1$
\vskip0.1truecm\noindent
{\it polynomial growth} if $(\E,\F)$ has discrete spectrum and, for some $c,d>0$ and all $n\in\mathbb{N}$,
\[
\beta_n\le c\cdot n^d.
\]
{\it intermediate growth} if it has subexponential growth but not polynomial growth.
\end{defn}\noindent
It is easy to see that the subexponential growth property can be formulated in terms of {\it nuclearity} of the Markovian semigroup $\{e^{-tL}: t>0\}$ on $L^2 (N,\o)$:
\begin{lem}
The Dirichlet form $(\E,\F)$ has discrete spectrum and subexponential spectral growth rate if and only if the Markovian semigroup $\{e^{-tL}: t>0\}$ on $L^2 (N,\o)$ is {\it nuclear, or trace-class}, in the sense that:
\[
{\rm Trace\, }(e^{-tL})<+\infty\qquad t>0\, .
\]
\end{lem}\noindent
Here is the announced connection between amenability and spectral properties.
\begin{thm}([CS5]).
Let $N$ be a $\sigma$-finite von Neumann algebra. If there exists a normal, faithful state $\o\in M_{*+}$ and a Dirichlet form $(\E,\F)$ on $L^2(N,\o)$ having subexponential spectral growth, then $N$ is amenable.
\end{thm}\noindent
Let us sketch the main points of the proof assuming, to simplify, that $\E[\xi_\o]=0$. Let $N\overline{\otimes} N^\circ$ the {\it von Neumann spatial tensor product} of $N$ and $N^\circ$. It turns out that the coarse representation of $N\otimes_{\rm max}N^\circ$, defined by
\[
\begin{split}
&\pi_{\rm co}:N\otimes_{\rm max}N^\circ\rightarrow \mathcal{B}(L^2(N,\o)\otimes L^2(N,\o)) \\
\pi_{\rm co} (x\otimes y^o)(\xi\otimes\eta)&:=x\xi\otimes\eta y\qquad x, y\in N\, ,\quad \xi, \eta\in L^2(N,\o)\, , \\
\end{split}
\]
give rise to the spatial tensor product of the von Neumann algebras

\[
(\pi_{\rm co}(N\otimes_{\rm max}N^\circ))'' = N\overline{\otimes} N^\circ .
\]
Moreover, the normal extension of the coarse representation $\pi_{\rm co}$ of $N\otimes_{\rm max}N^\circ$ to $N\overline{\otimes} N^\circ$ is the standard representation of $N\overline{\otimes} N^\circ$ so that
\[
L^2(N,\o)\otimes L^2(N,\o)\simeq L^2(N\overline{\otimes} N^\circ,\o\otimes\o^\circ)
\]
and the positive cone $L^2_+(N\overline{\otimes} N^\circ,\o\otimes\o^\circ)$ can be identified with the cone of all positive, Hilbert-Schmidt operators on $L^2(N,\o)$. In particular, since, by assumption, $e^{-tL}$ is a positive, Hilbert-Schmidt operator for all $t>0$, we have
\[
e^{-tL}\in L^2_+(N\overline{\otimes} N^\circ,\o\otimes\o^\circ)\qquad t>0.
\]
Since $\E$ is a complete Dirichlet form, its associated semigroup is completely positive and this implies that the linear functional $\Phi_t$, determined by
\[
\Phi_t :N\otimes_{\rm max} N^\circ\rightarrow\mathbb{C}\qquad \Phi_t(x\otimes y^\circ):=(i_\o (y^*)|e^{-tL} i_\o (x)),
\]
is positive and actually a state since $\E[\xi_\o]=0$. By the continuity properties of the symmetric embeddings and the above identifications, we have
\[
\Phi_t(z)=\Bigl(e^{-tL}\Bigl|i_{\o\otimes \o^\circ} (z)\Bigr)_{L^2(N\overline{\otimes} N^\circ,\o\otimes\o^\circ)}\qquad z\in N\overline{\otimes} N^\circ\, .
\]
Since $i_{\o\otimes \o^\circ}$ is positive preserving and $e^{-tL}$ is a positive element of the standard cone, we have that $\Phi_t$ is a normal state on $N\overline{\otimes} N^\circ$ and can thus be represented by a unique positive unit vector $\Omega_t\in L^2_+(N\overline{\otimes} N^\circ,\o\otimes\o^\circ)$ as
\[
\Phi_t(z)=\Bigl(\Omega_t |\pi_{\rm co}(z)\Omega_t\Bigr)_{L^2(N\overline{\otimes} N^\circ,\o\otimes\o^\circ)}\qquad z\in N\overline{\otimes} N^\circ.
\]
By the strong continuity of the Markovian semigroup $e^{-tL}$ on $L^2(N,\o)$, we then have
\[
\lim_{t\downarrow 0} \Bigl(\Omega_t |\pi_{\rm co}(z)\Omega_t\Bigr)_{L^2(N,\o)\otimes L^2(N,\o)}=(\xi_\o |\pi_{\rm id}(z)\xi_\o)=1\qquad z\in N\otimes_{\rm max}N^\circ\, .
\]
This proves that the trivial representation $\pi_{\rm id}$ of $N\otimes_{\rm max}N^\circ$, given by $\pi_{\rm id}(z):=I_{L^2(N,\o)}$ for all $z\in N\otimes_{\rm max}N^\circ$, is weakly contained in the coarse representation $\pi_{\rm co}$ and thus $N$ is amenable by a characterization of amenability due to S. Popa [Po].
\vskip0.2truecm
This approach by correspondences to relate spectral properties of Dirichlet forms to approximation properties of von Neumann algebras allows to treat also the {\it relative} case in which one deals with embeddings of subfactors $B\subset N$ on one side and with the a subexponential spectral growth rate of Dirichlet form r{\it elative to the subalgebra $B$}, on the other side. In these situations the essential spectrum of Dirichlet forms is not empty. (See [CS5]).

\subsection{Haagerup approximation property and discrete spectrum of Dirichlet forms}
The free group of two generators $\mathbb{F}_2$ is non amenable but in 1979 U. Haagerup proved in [H3] that its word-length function $\it l$ is conditionally negative definite and proper. This allowed him to prove that the group von Neumann algebra $L(\mathbb{F}_2)$ and the group C$^*$-algebra of $\mathbb{F}_2$ have the Grothendieck Metric Approximation Property. Moreover, the above properties of the length function of free groups also determine the following properties. This specific case opened the study of the following class of groups, larger than the class of amenable ones.
\begin{defn}
A countable, discrete group $\Gamma$ is said to have the {\it Haagerup Approximation Property} if there exists a sequence of positive definite functions in $c_0(\Gamma)$, uniformly convergent on compact subsets, to the constant function $1$ (see for example [CCJJV]). This property is equivalent to the existence of a {\it proper, conditionally negative definite function on $\Gamma$}.
\end{defn}\noindent
Clearly all amenable groups have the Property (H).
A series of contribution [Ch], [CS1,2], [COST] [J], allowed to isolate the following property of von Neumann algebras that for group algebras $L(\Gamma)$ of discrete groups is equivalent to the Haagerup Approximation Property of $\Gamma$.
\begin{defn}
A von Neumann algebra with standard form $(M,L^2(M),L^2_+(M),J)$ is said to have the {\it Haagerup Approximation Property} (HAP) if there exists a sequence of completely positive contractions $T_k :L^2(M)\to L^2(M)$ strongly converging to the identity operator
\[
\lim_k\|\xi-T_k\xi\|_{L^2(M)}=0\qquad \xi\in L^2(M).
\]

\end{defn}
Here is the announced connection between (HAP) and spectral properties.

\begin{thm}\label{propH}([CS1]).
Let $N$ be a $\sigma$-finite von Neumann algebra. Then the following properties are equivalent
\vskip0.1truecm\noindent
i) $M$ has the Property (HAP)
\vskip0.1truecm\noindent
ii) there exists on $L^2(M)$ a completely Dirichlet form $(\E,\F)$ with respect to some faithful, normal state $\o\in M_{*+}$,  having discrete spectrum.
\end{thm}
\begin{rem}
i) Property (H) may be formulated in a number of slightly different, equivalent ways also for not necessarily $\sigma$-finite von Neumann algebras too in such a way that the above spectral characterization remains anyway true.\\
ii) The construction of Dirichlet forms out of negative definite functions on groups and the above characterization of the Haagerup Approximation Property, indicate that Dirichlet forms for arbitrary von Neumann algebras play a role parallel to the one played by the continuous, negative definite functions on groups (see discussion in [CS1]).\\
%iii) It should be interesting to check this idea versus the class of groups having the Kazdhan Property (T) (by which all continuous, negative definite functions are bounded).
\end{rem}

\subsection{Property ($\Gamma$) and Poincar\'e inequality for elementary Dirichlet forms}
Another instance of the interactions among structural properties of a von Neumann algebra $M$ and spectral properties of Dirichlet forms on $L^2(M)$ may be shown reformulating the Murray-von Neumann Property ($\Gamma$).
\vskip0.2truecm\noindent
By an {\it elementary} completely Dirichlet form on a finite von Neumann algebra $(M,\tau)$, endowed with a normal, trace state, we mean one of type
\[
\E_F[\xi]:=\sum_{x\in F}\|x\xi-\xi x\|^2_{L^2(M,\tau)}\qquad \xi\in L^2(M,\tau)
\]
for some finite, symmetric set $F=F^*\subset M$. The unit, cyclic vector $\xi_\tau\in L^2(M,\tau)$ representing the trace is central so that $\E_F[\xi_\tau]=0$ and $\lambda_0=0$ is an eigenvalue for all elementary Dirichlet forms. Elementary Dirichlet forms are everywhere defined and thus bounded.
\begin{defn}([Co1]).
A finite von Neumann algebra endowed with its normal, tracial state $(M,\tau)$, has the Property ($\Gamma$) if for any $\varepsilon >0$ and any finite set $F\subset M$ there exists a unitary $u\in M$ with $\tau (u)=0$ such that $\|(ux-xu)\xi_\tau\|_2<\varepsilon$ for all $x\in F$.
\end{defn}
This property was the first invariant introduced by  F.J. Murray and J. von Neumann [MvN] to show the existence of non hyperfinite $II_1$ factors. For example, the group von Neumann algebra $L(S_\infty)$ of the countable discrete group $S_\infty$ of finite permutations of a countable set and the Clifford von Neumann algebra of a separable Hilbert space are both isomorphic to the hyperfinite $II_1$ factor $R$, which fullfill the Property ($\Gamma$). This latter cannot be isomorphic to the group algebra $L(\mathbb{F}_n)$ of the free group $\mathbb{F}_n$ with $n\ge 2$ generators which is a $II_1$ factor but does not have the Property ($\Gamma$) (and in fact it is not hyperfinite).
\vskip0.2truecm\noindent
It is well known [Po] that the absence of the Property ($\Gamma$) for a $II_1$ factor with separable predual, is a {\it rigidity property} equivalent to the existence of a spectral gap for suitable self-adjoint, finite, convex combinations of inner automorphisms, as unitary operators on $L^2(M,\tau)$.\\
We now show how the Property ($\Gamma$) can be also naturally interpreted in terms of a spectral property of elementary Dirichlet forms.
\begin{thm}([CS6]).
A finite von Neumann algebra endowed with its normal, tracial state $(M,\tau)$, has the Property $(\Gamma )$ if and only if for any elementary completely Dirichlet form
\[
\E_F[\xi]=\sum_{x\in F}\|x\xi-\xi x\|^2_{L^2(M,\tau)}\qquad \xi\in L^2(M,\tau),
\]
associated to a finite set $F=F^*\subset M$, the eigenvalue $\lambda_0:=0$ is not isolated in the spectrum.
\vskip0.2truecm\noindent
Otherwise stated, $(M,\tau)$, does not have the Property $(\Gamma)$ if and only if there exists an elementary Dirichlet form $\E_F$ such that the eigenvalue
$\lambda_0=0$ is isolated (spectral gap) or, equivalently, $\E_F$ satisfies, for some $c_F>0$, a Poincar\'e inequality
\[
c_F\cdot\|\xi-(\xi_\tau|\xi)\xi_\tau\|^2_2\le \E_F[\xi]\qquad \xi\in L^2(M,\tau).
\]
\end{thm}
\begin{proof}
If $J$ denotes the symmetry on $L^2(M,\tau)$ which extends the involution of $M$, then for $u,x\in M$, setting $\xi:=x\xi_\tau\in M\xi_\tau$, we have $(ux-xu)\xi_\tau=u\xi-Ju^*x^*\xi_\tau=u\xi-Ju^*Jx\xi_\tau=u\xi-\xi u$. Since $\xi_\tau \in L^2(M,\tau)$ is cyclic, i.e. $M\xi_\tau$ is dense in $L^2(M,\tau)$, if $(M,\tau)$ does not have the Property $\Gamma$ there exists $\varepsilon>0$ and an elementary Dirichlet form $\E_F$ such that for all unitaries $u\in M$ we have
\[
\varepsilon\cdot \|u\xi_\tau-\tau(u)\xi_\tau\|^2_2\le \E_F[u\xi_\tau].
\]
For any $y=y^*\in M$ such that $\|y\|_M\le 1/\sqrt{2}$, consider the unitaries $u_\pm :=y\pm i\sqrt{1_M-y^2}$ so that $y=(u_++u_-)/2$. Since $\sqrt{1-y^2}=\phi(y)$ with $\phi(s):=\sqrt{1-s^2}$ and $|\phi'(s)|\le 1$ for $|s|\le 1/\sqrt{2}$, it follows by the Markovianity of the Dirichlet form that $\E_F[\sqrt{1-y^2}\xi_\tau]\le \E_F[y\xi_\tau]$. Since $\E_F$ is $J$-real we then have
\[
\E_F[u_\pm\xi_\tau]=\E_F[(y\pm i\sqrt{1-y^2})\xi_\tau]=\E_F[y\xi_\tau]+\E_F[\sqrt{1-y^2}\xi_\tau]\le 2\E_F[y\xi_\tau]
\]
and
\[
\varepsilon\cdot (\|u_+\xi_\tau-\tau(u_+)\xi_\tau\|^2_2 + \|u_-\xi_\tau-\tau(u_-)\xi_\tau\|^2_2)\le \E_F[u_+\xi_\tau] + \E_F[u_-\xi_\tau]\le 4\E_F[y\xi_\tau].
\]
Thus for all $y\in M$ we have
\[
\begin{split}
\|y\xi_\tau-\tau(y)\xi_\tau\|^2_2&=\|u_+\xi_\tau-\tau(u_+)\xi_\tau+u_-\xi_\tau -\tau(u_-)\xi_\tau\|^2_2 \\
&\le 2\bigl(\|u_+\xi_\tau-\tau(u_+)\xi_\tau\|^2_2 + \|u_-\xi_\tau-\tau(u_-)\xi_\tau\|^2_2\bigr) \\
&\le 8\varepsilon^{-1}\cdot\E_F[y\xi_\tau].
\end{split}
\]
and, by the density of $M\xi_\tau$ in $L^2(M,\tau)$, a Poincar\'e inequality holds true with $c_F=\varepsilon/8$.
\end{proof}\noindent
By classical results of A. Connes [Co3], obtained along his classification of injective factors, one can relate the existence of spectral gap for an elementary Dirichlet form to fundamental properties of $II_1$ factors $(M,\tau)$ with separable predual: the following properties are equivalent
\vskip0.1truecm\noindent
i) the subgroup ${\rm Inn}(M)$ of inner automorphisms is closed in ${\rm Aut}(M)$ ($M$ is called a {\it full factor})
\vskip0.1truecm\noindent
ii) the C$^*$-algebra $C^*(M,M')$ generated by $M$ and its commutant $M'$, acting standardly on $L^2(M,\tau)$, contains the ideal of compact operators
\vskip0.1truecm\noindent
iii) there exists an elementary Dirichlet form $\E_F$ on $L^2(M,\tau)$ satisfying a Poincar\'e inequality.
\subsection{Property (T)}
Groups having the Kazdhan property (T) show, in many instances, a very rigid character. By their definition, all continuous, negative definite functions on them are bounded (see [CCJJV]) and they can be characterized by any of the following properties: i) whenever a sequence of continuous, positive definite functions converges to 1 uniformly on compact subsets, then it converges uniformly ii) if a representation contains the trivial representation weakly, then it contains it strongly iii) every continuous, isometric action on an affine Hilbert space has a fixed point.\\
In von Neumann algebra theory, the Property (T) of a group $\Gamma$ with infinite conjugacy classes, were first considered by A. Connes to show that the factor $L(\Gamma)$ has a countable fundamental group. The same author characterized countable, discrete groups $\Gamma$ having the Property (T) through a specific property of $L(\Gamma)$. Later A. Connes and V. Jones [CJ] identified a property (T) for general von Neumann algebras in strong analogy with one of the above characterizations for the groups case. They key point was the replacement of the notion of group representation by that of correspondence for general von Neumann algebras:\\
$M$ {\it has the property (T) if all correspondences sufficiently close to the standard one must contain it}. \\
In Section 10 below, we will describe a recent result by A. Skalski and A. Viselter to a Dirichlet form characterization of the property (T) of von Neumann algebras of locally compact quantum groups.

%\begin{thm}
%Let $(M,\tau)$ be a finite $II_1$ factor with property (T). Then any completely Dirichlet form on $L^2(M,\tau)$ has non empty essential spectrum.
%\end{thm}

\section{KMS-symmetric semigroups on C$^*$-algebras}
We have seen that the extension of the theory of Dirichlet forms introduced by S. Albeverio and R. Hoegh-Khron and developed and applied by J.-L. Sauvageot [S3,7,8] and by E.B. Davies [D2], E.B. Davies and O. Rothaus [DR1,2] and by E.B. Davies and M. Lindsay [DL1,2], can be applied to several fields in which the relevant algebra of observables, to retain a physical language, is no more commutative. This theory concerns, however, C$^*$-algebras or von Neumann algebras endowed with a well behaved trace functional. To have a theory suitable to be applied to other fields one has to face the problem to give a meaning to Markovianity of Dirichlet forms with respect non tracial states. For example,\\
i) equilibria in Quantum Statistical Mechanics or Quantum Field Theory are described by states obeying the Kubo-Martin-Schwinger condition which are not trace at finite temperature
ii) in Noncommutative Geometry the algebra generated by the "coordinate functions" of a noncommutative space may have a natural relevant state which is not a trace, as it is the case of the Haar state of several Compact Quantum Groups.
\vskip0.1truecm\noindent
In this section we describe this extension of the theory of Dirichlet forms which deals with Markovianity with respect to KMS states on C$^*$-algebras and with any normal, faithful states on von Neumann algebras. In the next sections we shall have occasion to describe applications were this generalized theory is due.

\subsection{KMS-states on C$^*$-algebras}
Let $A$ be a C$^*$-algebra and let $\{\alpha_t:t\in\R\}$ be a strongly continuous automorphism group on it, often interpreted as a dynamical system.
\begin{defn} ({\bf KMS-states})([Kub], [HHW]).
Let ${\bf \alpha}:=\{\alpha_t : t\in \mathbb{R}\}$ be a strongly continuous group of automorphisms of a C$^*$-algebra $A$ and $\beta\in \mathbb{R}$.  A state $\omega$ is said to be a $({\bf \alpha}, \beta )$-KMS state if it is $\alpha$-invariant and if the following {\it KMS-condition} holds true:
\par\noindent
%\begin{equation}
\[
\o (a\alpha_{i\beta}(b))=\o (ba)
\]
%\end{equation}
for all $a,b$ in a norm dense, ${\bf \alpha}$-invariant $^*$-algebra of analytic element for $\alpha$. If $M$ is a von Neumann algebra and ${\bf \alpha}:=\{\alpha_t : t\in \mathbb{R}\}$ is a $w^*$-continuous group of automorphisms, a state $\omega$ is said to be a $({\bf \alpha}, \beta )$-KMS state if $\o$ is $\alpha$-invariant, normal and the KMS-condition above holds true for all $a,b$ in a $\sigma (M ,M_*)$-dense, ${\bf \alpha}$-invariant $^*$-subalgebra of $A_\alpha$. KMS states corresponding to $\beta =0$ are just the traces over $M$.\\
Notice that  any faithful normal state $\o$ on a von Neumann algebra $M$ is a $(\sigma^\o,-1)$-KMS state, i.e. a KMS state for the modular group $\sigma^\o$ at inverse temperature $\beta=-1$. In this case, in fact, the KMS condition coincides with modular condition.
\end{defn}
\begin{defn}({\bf KMS-symmetric Markovian semigroups on C$^*$-algebras})([Cip2,5]).
Let ${\bf \alpha }:=\{\alpha_t : t\in \mathbb{R}\}$ be a strongly continuous group of automorphisms of a C$^*$-algebra $A$ and $\o$
be a fixed $({\bf \alpha } ,\beta )$-KMS state, for some $\beta\in \mathbb{R}$.\\
A bounded map $R :A\to A$ is said to be $({\bf \alpha } ,\beta )$-{\it KMS symmetric with respect to} $\o$ if
\begin{equation}
\omega \Bigl(b R(a)\Bigr)=\omega \Bigl(\alpha_{-\frac{i\beta}{2}}(a) R(\alpha_{+\frac{i\beta}{2}}(b))\Bigr)
\end{equation}
for all $a,b$ in a norm dense, ${\bf \alpha}$-invariant $^*$-algebra of analytic elements for $\alpha$.\\
A strongly continuous semigroup $\{R_t : t\ge 0 \}$ on  $A$ is said to be $({\bf \alpha } ,\beta )$-{\it KMS symmetric with respect to} $\o$ if
$R_t$ is $({\bf \alpha } ,\beta )$-KMS symmetric with respect to $\o$ for all $t\ge 0$.\\
In the von Neumann algebra case, $\o$ is assumed to be normal, maps and semigroups to be point-weak*-continuous and the subalgebra $B$ to be weak*-dense.
\end{defn}
Let ${\bf \alpha }:=\{\alpha_t : t\in \mathbb{R}\}$ be a strongly continuous group of automorphisms of a C$^*$-algebra $A$ and $\o$
be a fixed $({\bf \alpha } ,\beta )$-KMS state, for some $\beta\in\mathbb{R}$. Let $(\pi_\o ,\H_\o , \xi_\o )$ be the corresponding GNS-representation, $\widehat{\o}$ the normal extension of $\o$ to the von Neumann algebra $M:=\pi_\o (A)^{\prime\prime}$ and $\widehat{{\bf \alpha }}:=\{\widehat{\alpha}_t : t\in \mathbb{R}\}$ be the induced weak*-continuous group of automorphisms of $M$. Comparing the KMS condition for $\widehat\o$ with respect to $\widehat\alpha$ to its modular condition, one  readily observe that the modular group of $\widehat\o$ is given by
\[
\sigma^{\widehat\o}_t=\widehat\alpha_{-\beta t}\qquad t\in\R.
\]
The following is a key consequence of the $(\alpha,\beta)$-KMS-symmetry of a map.
\begin{lem} ([Cip2])
A map $R :A\to A$ which is $({\bf \alpha } ,\beta )$-KMS symmetric with respect to $\o$, leaves globally invariant the kernel $\mathrm{ker} (\pi_\o )$
of the GNS-representation of $\o$.
\end{lem}\noindent
This result allows to study KMS symmetric maps and semigroups on the von Neumann algebra associated to the GNS representation of the KMS state.
\begin{thm}(von Neumann algebra extension of KMS-symmetric semigroups)([Cip2])\\
Let $\{R_t:t\ge 0\}$ be a strongly continuous semigroup on $A$, $(\alpha,\beta)$-KMS symmetric with respect to $\o$.
Then there exists a unique point-weak*-continuous semigroup  $\{S_t : t\ge 0\}$ on $M$
determined by
\begin{equation}
S_t (\pi_\o (a))=\pi_\o (R_t (a))\, ,\quad a\in A\, ,\qquad t\ge 0\, .
\end{equation}
This extension is $\widehat{\o}$-modular symmetric in the sense that
\begin{equation}
\hat{\o} \Bigl(S_t (x)\sigma^{\widehat\o}_{-\frac{i}{2}}(y)\Bigr)=\widehat{\o} \Bigl(\sigma^{\widehat\o}_{+\frac{i}{2}}(x)S_t (y)\Bigr)\qquad t\ge 0\, ,
\end{equation}
for all $x,y$ in a weak*-dense, $\sigma^{\widehat\o}$-invariant $^*$-algebra of analytic elements $\sigma^{\widehat\o}$. Moreover, if $\{R_t : t\ge 0\}$ is positive, completely positive, Markovian or completely Markovian, then $\{S_t : t\ge 0\}$ shares the same properties.
\end{thm}
As a consequence, a Dirichlet form on $L^2(A,\o)$ is determined by the semigroup on $A$
\begin{cor}
Let $(L,D(L))$ be the generator of the semigroup $\{R_t:t\ge 0\}$ on $A$. Then the Dirichlet form on $L^2(A,\o)$ associated to the  strongly continuous extension of the $\widehat\o$-modular symmetric semigroup $\{S_t : t\ge 0\}$ on $M$, satisfies the relation
\[
\E[i_\o(\pi_\o(a))]=(i_\o(\pi_\o(a))|i_\o(\pi_\o(La)))_{L^2(A,\tau)}\qquad a\in D(L).
\]
\end{cor}
By this result one may study properties of the semigroup $R$ on the C$^*$-algebra through the associated Dirichlet form $\E$ on $A$. Notice that by definition we have the coincidence of the spaces $L^2(A,\o)=L^2(M,\widehat\o)$.\\
This result suggests also that one can approach the construction of Markovian semigroups $(\alpha,\beta)$-symmetric with respect to a $(\alpha,\beta)$-KMS state $\o$ on a C$^*$-algebra $A$, through the construction of Dirichlet forms on $L^2(A,\o)$. The advantage being that working with quadratic forms instead that linear operators often allows to relax domain constrains to prove closability. To finalize this approach, however, once obtained from the Dirichlet form on $L^2(A,\tau)$ the Markovian semigroup on the von Neumann algebra $L^\infty(A,\tau)$, one has to face the problem to show that the C$^*$-algebra $A$ is left invariant and that on it the semigroup is not only w$^*$-continuous but in fact strongly continuous. This last problem may be solved case by case as we did for the Ornstein-Uhlenbeck semigroup in [CFL] for example. We notice, however, that even in classical potential theory, on Riemmanian manifolds the construction of the heat semigroup on the algebra of continuous functions requires a certain amount of substantial potential analysis ([D2]).
\section{Application to Quantum Statistical Mechanics}
After the proof, in the early nineties of the last century, by D. Stroock and B. Zegarlinski, of the equivalence between the Dobrushin-Shlosman mixing condition and the uniform logarithmic Sobolev inequalities for classical spin systems with continuous spin space, efforts were directed to obtain for quantum spin systems similar results, within the framework of the studies of the convergence to equilibrium. See for example [LOZ], [MZ1], [MZ2], [MOZ], [Mat1], [Mat2], [Mat3].\\
In this section we describe just one of these constructions of Markovian semigroups by Dirichlet forms for KMS states of quantum spin systems, provided by Y.M. Park and his school [P1], [P2], [P3], [BKP1], [BKP2].

\subsection{Heisenberg Quantum Spin Systems.}
Let us describe briefly, the quantum spin system and its dynamics. The observables at sites of the lattice $\Z^d$ are elements of the algebra $M_2(\mathbb{C})$ and the C$^*$-algebra of observables confined in the finite region $X\subset\Z^d$ is
\[
A_X :=\bigotimes_{x\in X} M_2(\mathbb{C}).
 \]
If $\L$ denotes the net of all finite subsets of $ \Z^d$, directed by inclusion, the system $\{A_X:X\in\L\}$ is in a natural way a net of C$^*$-algebras so that the algebra of all {\it local observables} given by
\[
A_0 :=\bigcup_{X\in\L} A_X\, ,
\]
is naturally normed and its norm completion is a C$^*$-algebra $A$ (quasi-local observables).
\vskip0.2truecm\noindent
Interactions among particles  in finite regions is represented by a family of self-adjoint elements $\Phi:=\{\Phi _X :X\in\L\}\subset A_X$. Using the Pauli's matrices $\sigma^x_j\in M_2(\mathbb{C})$, $j=0,1,2,3,$ at the sites $x\in \mathbb{Z}^d$, in the {\it isotropic, translation invariant, Heisenberg model}, for example, in addition to an external potential represented by a one-body interaction of strength  $h\in\R$
\[
\Phi (\{x\}):=h\sigma_3^x,
\]
particles interact only by a two-body potential so that $\Phi (X)=0$ whenever $|X|\ge 3$ and
\[
\Phi (\{x,y\}):=J(x-y)\sum_{i=1}^3 \sigma_i^x\sigma_i^y\qquad x\neq y
\]
for a parameter $\lambda>0$ and a function $J:\Z^d\to\R$ describing the strength of the interaction between pairs of particles, under the assumption
\[
 \sum_{x\in\Z^d}e^{\lambda |x|} |J (x)|<+\infty.
 \]
For any fixed $Y\in\L$, the derivation
\[
A_0\ni a\mapsto i\big[\Phi_Y ,a\big]\in A
\]
extends to a bounded derivation on $A$ and generates a uniformly continuous group of automorphisms of $A$, representing the time evolution of the observables, interacting with those particles confined in $Y$. To take into account simultaneously, the mutual influences among particles in different regions, one verifies that the superposition
\begin{equation}
D(\delta ):=A_0\qquad
\delta (a):= \sum_{Y\in \mathcal{L}} i\big[\Phi_Y ,a\big]\, ,
\end{equation}
is a closable derivation on $A$ whose closure is the generator of a strongly continuous group ${\bf \alpha}^\Phi :=\{\alpha_t^\Phi :t\in\R\}$ of automorphisms of $A$.

\subsection{Markovian approach to equilibrium}
The above interactions provide the existence of $({\bf \alpha}^\Phi\, ,\beta )$-KMS-states $\o$ at any inverse temperature $\beta >0$. Let $(\pi_\o , \H_\o ,\xi_\o)$ be the GNS representations of the state $\o$, $M$ the von Neumann algebra $\pi_\o (A)^{\prime\prime}$ and
$(M ,L^2(A,\o), L^2_+(A,\o), J_\o)$ the corresponding standard form. In the following we will use the smearing function $f_0(t):=1/\cosh(2\pi t)$.
\begin{thm}
Suppose that $\o$ is a $({\bf \alpha}^\Phi\, ,\beta )$-KMS-state at an inverse temperature satisfying
\begin{equation}
\beta<\frac{\lambda}{\|\Phi\|_\lambda}
\end{equation}
where
\begin{equation}
\|\Phi\|_\lambda :=\sup_{x\in\Z^d}\sum_{x\in X\in\mathcal{L}}|X|4^{|X|}e^{\lambda D(X)}\|\Phi _X\|_{A_X}
\end{equation}
is finite under the exponential decay assumption on the strength $J$. Then the quadratic forms associated to the self-adjoint elements $a^x_j :=\pi_\o (\sigma^x_j)$
\begin{equation}
\E_{x,j}[\xi] = \int_\R \|(\sigma_{t-i/4}(a^x_j) - j(\sigma_{t-i/4}(a^x_j)))\xi\|^2\, f_0(t)dt
\end{equation}
are bounded completely Dirichlet forms and
\begin{equation}
\E:L^2(A,\o)\to [0,+\infty]\qquad \E[\xi]:=\sum_{x\in\Z^d}\sum_{j=0}^3 \E_{x,j}[\xi]
\end{equation}
is a completely Dirichlet form on $L^2(A,\o)$.
\end{thm}
Concerning the proof, a first observation is that  the $\E_{x,j}$ are bounded Dirichlet form as uniformly convergent continuous superposition of elementary completely Dirichlet forms. The quadratic form $\E$ is Markovian and closed as pointwise monotone limit of bounded completely Dirichlet forms. The only point that is left to be shown is  the fact that it is densely defined, i.e. $\E$ is finite on a dense domain in $L^2(A,\o)$. This is a consequence of the fact that under the current hypotheses on the strength on the interaction, the dynamics has {\it finite speed propagation} in the sense that, denoting by $d(x,X)$ distance of the site $x\in \mathbb{Z}^d$ from the region $X\in \mathcal{L}$, we have
\begin{equation}
\|[\alpha_t^\Phi (a),b]\|\le 2\|a\|\cdot\|b\|\cdot|X|\cdot e^{-\bigl(\lambda d(x,X))-2|t|\|\Phi\|_\lambda \bigr)}\qquad a\in A_{\{x\}}\, ,b\in A_X\, ,t\in\mathbb{R}\, .
\end{equation}

Concerning the ergodic behaviour of the semigroups associated to the Dirichlet forms above, the following result shows how these properties are deeply connected to the other fundamental properties of the KMS-state.

\begin{cor} ([P2 Theorem 2.1])
Within the assumption of Theorem 3.3, the following properties are equivalent:
\vskip0.2truecm\noindent
\item i) $\o$ is an extremal $({\bf \alpha}^\Phi\, ,\beta )$-KMS-state;
\vskip0.2truecm\noindent
\item ii) $\o$ is a factor state in the sense that the von Neumann algebra $M:=\pi_\o(A)''$ is a factor.
\vskip0.2truecm\noindent
\item iii) the Markovian semigroup $\{T_t : t\ge 0\}$ is ergodic in the sense that the subspace of $L^2(A,\o)$ where it acts as the identity operator is reduced to the scalar multiples of the cyclic vector $\xi_\o\in L^2(A,\o)$ representing the KMS state $\o$.
\end{cor}\noindent
Extremality, i.e. the impossibility to decompose a KMS state as convex, nontrivial superposition of other KMS states (see [BR2]), is the mathematical translation of the notion of {\it pure phase} in Statistical Mechanics.\\
Ergodicity of Markovian semigroups were considered by L. Gross [G1] to prove the uniqueness of the ground state of physical Hamiltonians in Quantum Field Theory. Later, S. Albeverio and R.Hoegh-Khron [AHK2] established a Frobenious type theory for positivity preserving maps on von Neumann algebras with trace and in [Cip3] a Perron type theory was provided for positivity preserving maps on the standard form of general von Neumann algebras.

\section{Applications to Quantum Probability}

As pointed out in Introduction, one of the major achievement of the theory of commutative potential theory is the correspondence between regular Dirichlet forms and symmetric Markov-Hunt processes on metrizable spaces. In noncommutative potential theory we do not dispose at moment of a complete theory but we have at least a clear connection between Dirichlet forms of translation invariant, symmetric, Markovain semigroups and L\'evy's Quantum Stochastic Processes on Compact Quantum Groups.
\subsection{Compact Quantum Groups d'apres S.L. Woronowicz} ([W]).
In the following $m_A:A\otimes_{\rm alg}A\to A$ will denote the extension of the product operation of $A$.\\
Let us recall that a compact quantum group $\mathbb{G}:=(A,\Delta)$ is a unital C$^*$-algebra $A =: C(\mathbb{G})$ together with a\\
i) {\it coproduct} $\Delta : A \to A \otimes_{\rm max} A$, a unital, $^*$-homomorphism which is\\
ii) {\it coassociative} $(\Delta  \otimes id_A) \circ\Delta = (id_A \otimes \Delta) \circ \Delta$ and satisfies\\
iii) {\it cancellation rules}: the closed linear span of $(1 \otimes A)\Delta(A)$ and $(A \otimes 1)\Delta(A)$ is $A \otimes A$.
\vskip0.2truecm\noindent
An example of the above structure arise from a compact group $G$ by dualization of its structure. In fact, setting $A:=C(G)$ we have $A\otimes_{\rm max} A=C(G\times G)$ and a coproduct defined by
\[
(\d f)(s,t):=f(st)\qquad f\in C(G),\quad s,t\in G.
\]
\vskip0.2truecm\noindent
A {\it unitary co-representation} of $\mathbb{G}$ is a unitary matrix $U = [u_{jk}] \in M_n(A)$ such that
\[
\d u_{jk}=\sum_{i=1}^n u_{ji}\otimes u_{ik}\qquad j,k=1,\cdots,n.
\]
Denote by $\widehat{\mathbb{G}}$ the set of all equivalence classes of unitary co-representations of $\mathbb{G}$. If a  family of inequivalent irreducible, unitary co-representations $\{U^s:s\in\mathbb{G}\}$ of $\mathbb{G}$ exhausts all of  $\widehat{\mathbb{G}}$, then the algebra of {\it polynomials}, defined by the linear span of the coefficients of all unitary co-representations
\[
{\rm Pol}(\mathbb{G}):={\rm linear\,\, span}\{u_{jk}\in A:[u_{jk}]\in \widehat{\mathbb{G}}\}
\]
is a {\it Hopf $^*$-algebra}, dense in $A$, with {\it counit} $\epsilon$ and {\it antipode} $S$ determined by
\[
\epsilon(u_{jk}):=\delta_{jk},\quad S(u_{jk}):=u_{kj}^*\qquad [u_{jk}]\in\widehat{\mathbb{G}}
\]
and satisfying the rules
\[
(\epsilon\otimes id)\Delta (a)=a,\qquad (id\otimes \epsilon)\Delta(a)=a,\qquad m_A(S\otimes id)\Delta(a)=\epsilon(a)1_A=m_A(id\otimes S)\Delta(a).
\]
The C$^*$-algebra $C(\mathbb{G})$ of a compact quantum group $\mathbb{G}$ is commutative if and only if it is of the form $C(G)$ for some compact group $G$. In this case counit and antipode are defined by
\[
\epsilon(f):=f(e),\qquad S(f)(s):=f(s^{-1})\qquad s\in G,
\]
where $e\in G$ is the group unit.
\vskip0.2truecm\noindent
Combining the tensor product with the coproduct, one may introduce new operations that in the case of compact group reduce to the well known classical ones.
\vskip0.2truecm\noindent
The {\it convolution} $\xi\ast\xi'\in A^*$ of {\it functionals} $\xi,\xi'\in A^*$ is defined by
\[
\xi\ast\xi':=(\xi\otimes\xi')\circ\Delta
\]
and the {\it convolution} $\xi\ast a\in A^*$ of a {\it functional} $\xi\in A^*$ and an element $a\in A$ is defined by
\[
\xi\ast a=(id\otimes \xi)(\Delta a)\qquad a\ast\xi:=(\xi\otimes id)(\Delta a).
\]
By a fundamental result of S.L. Woronowicz, on a compact quantum group $\mathbb{G}$ there exists a unique (Haar) state $h\in A^*_+$ which is both {\it left and right translation invariant} in the sense that
\[
a\ast h=h\ast a=h(a)1_A\qquad a\in A=C(\mathbb{G}).
\]
In the commutative case the Haar state reduces to the integral with respect to the Haar probability measure. However, in general, the Haar state is not even a trace but it is a $(\sigma, -1)$-KMS state with respect to a suitable automorphisms group $\sigma_t\in {\rm Aut}(A)$, $t\in\mathbb{R}$,
\[
h(ab)=h(\sigma_{-i}(b)a)\qquad a,b\in {\rm Pol}(\mathbb{G}).
\]
By a result of S.L. Woronowicz, the antipode $S:{\rm Pol}(\mathbb{G})\to C(\mathbb{G})$ is a densely defined, closable operator on $A$ and its closure $\bar S$ admits the polar decomposition
\[
\bar S=R\circ \tau_{i/2}
\]
where\\
i) $\tau_{i/2}$ generates a $^*$-automorphisms group $\tau:=\{\tau_t:t\in\R\}$ of the C$^*$-algebra $A$ and\\
ii) $R$ is a linear, anti-multiplicative, norm preserving involution on $A$ commuting with $\tau$, called {\it unitary antipode}.

\subsubsection{$SU_q(2)$ compact quantum group}
The compact quantum group $SU_q(2)$ with $q\in (0,1]$, is defined as the universal C$^*$-algebra  generated by the coefficients of a matrix
\[
U=
\begin{bmatrix}
\alpha & -q\gamma^*\\
\gamma & \alpha^*
\end{bmatrix}
\]
subject to the relations ensuring unitarity: $UU^*=U^* U=I$. Then one may check that, in terms of the generators $\alpha,\gamma$, all the other relevant objects are determined by
\vskip0.1truecm\noindent
i) comultiplication: $\Delta(\alpha):=\alpha\otimes\alpha +\gamma\otimes\gamma$, $\quad\Delta(\gamma):=\gamma\otimes\alpha + \alpha^*\otimes\gamma$\\
ii) counit: $\epsilon(\alpha):=1$, $\quad \epsilon(\gamma):=0$\\
iii) antipode: $S(\alpha):=\alpha^*$, $\quad S(\gamma):=-q\gamma$, $\quad S(u_{j,k}):=(-q)^{(j-k)}u_{-k,j}$ for $[u_{jk}]\in \widehat{\mathbb{G}}$\\
iv) Haar state: $h(u_{jk}):=0$ for $[u_{jk}]\in \widehat{\mathbb{G}}$\\
v) automorphisms group: $\sigma_z(u_{jk}):=q^{2iz(j+k)}u_{jk}$ for $[u_{jk}]\in \widehat{\mathbb{G}}$ and $z\in\mathbb{C}$\\
vi) unitary antipode: $R(u_{jk}):=q^{k-j}u_{jk}^*$  for $[u_{jk}]\in \widehat{\mathbb{G}}$.
\vskip0.1truecm\noindent
When $q=1$ one recovers the classical compact group $SU(2)$.
\subsubsection{Countable discrete groups as CQGs}
Let $\Gamma$ be a countable discrete group and $\lambda:\Gamma\to B(l^2(\Gamma))$ its left regular representation
\[
\lambda_s:l^2(\Gamma)\to l^2(\Gamma)\qquad \lambda_s(\delta_t):=\delta_{st}\qquad s,t\in\Gamma.
\]
The reduced C$^*$-algebra $C^*_r(\Gamma)\subset B(l^2(\Gamma))$ is the smallest C$^*$-algebra containing all the unitary operators $\lambda_s$ for $s\in\Gamma$. If instead of the regular representation one uses the direct sum of all cyclic unitary representation of $\Gamma$, the resulting algebra is called the universal C$^*$-algebra. It is isomorphic to the regular one if and only if $\Gamma$ is amenable.\\
A compact quantum group structure on C$^*_r(\Gamma)$ is obtained extending to a $^*$-homomorphism $\Delta$ from C$^*_r(\Gamma)$ to $C^*_r(\Gamma)\otimes C^*_r(\Gamma)$ the map defined by $\Delta(\lambda_s):=\lambda_s\otimes\lambda_s$ for $s\in\Gamma$. The linear span of the unitaries $\lambda_s$ for $s\in\Gamma$ is a dense $^*$-Hopf agebra on which counit and antipode are defined as  $\epsilon(\lambda_s)=1$ and $S(\lambda_s):=\lambda_{s^{-1}}$ for $s\in\Gamma$. The compact quantum group $C^*_r(\Gamma)$ is {\it cocommutative} in the sense that the comultiplication $\Delta$ is invariant under the flip of the left and right factors of $C^*_r(\Gamma)\otimes C^*_r(\Gamma)$. A theorem of Woronowicz ensures that any cocommutative compact quantum group $C(\mathbb{G})$ is essentially the C$^*$-algebra of a countable discrete group in the sense that there exists a countable discrete group $\Gamma$ and $^*$-homomorphisms $C^*(\Gamma)\to C(\mathbb{G})\to C^*_r(\Gamma)$. The CQG $C^*_r(\Gamma)$ is of Kac type and the Haar state coincides with the trace determined by $\tau(\delta_s)=0$ for $s\neq e$ and $\tau(\delta_e)=1$.

\subsection{L\'evy processes on Compact Quantum Groups}([CFK]).
The L\'evy processes on compact groups are among the most investigated stochastic processes in classical probability. We briefly describe in this section a class of quantum stochastic processes, in the sense of [AFL] (see also [GS]), on compact quantum groups that generalize the classical L\'evy processes.
\vskip0.2truecm
Let $(P,\Phi)$ be a von Neumann algebra with a faithful, normal state, also called a {\it noncommutative probability space}.
\vskip0.2truecm\noindent
i) A {\it random variable} on $\mathbb{G}$ is a $^*$-algebra homomorphism $j:{\rm Pol}(\mathbb{G})\to P$\\
ii) the {\it distribution} of the random variable is the state $\phi_j:=\Phi\circ j$ on ${\rm Pol}(\mathbb{G})$\\
iii) the {\it convolution} $j_1\ast j_2$ of the random variable $j_1,j_2:{\rm Pol}(\mathbb{G})\to P$ is the random variable
\[
j_1\ast j_2:=m_P\circ (j_1\otimes j_2)\circ \Delta
\]
where $m_P$ is the product in $P$.
\vskip0.2truecm\noindent
A Quantum Stochastic Process ([AFL]) is a family of random variables $\{j_{s,t}:0\le s\le t\}$ satisfying
\vskip0.1truecm\noindent
i) $j_{tt}=\epsilon 1_P$ for all $0\le t$\\
ii) {\it increment property}: $j_{rs}\ast j_{st}=j_{rt}$ for all $0\le r\le s\le t$\\
iii) {\it weak continuity}: $j_{tt}\to j_{ss}$ in distribution as $t\to s$ decreasing.
\begin{defn}(Quantum L\'evy Processes)
A L\'evy process on a CQG $\mathbb{G}$ is a quantum stochastic process on the Hopf-algebra ${\rm Pol}(\mathbb{G})$ such that it has
\vskip0.1truecm\noindent
i) {\it independent increments} in the sense that for disjoint intervals $(s_k,t_k]$, $k=1,\cdots ,n$
\[
\Phi(j_{s_1t_1}(a_1)\cdots j_{s_nt_1}(a_n))=\Phi(j_{s_1t_1}(a_1))\cdots \Phi(j_{s_nt_1}(a_n))
\]
ii) {\it stationary increments} in the sense that the distribution $\phi_{st}=\Phi\circ j_{st}$ depends only on $t-s$.
\end{defn}
\begin{thm}
Under a suitable probabilistic notion of equivalence of quantum stochastic processes, equivalence classes of L\'evy processes $\{j_{s,t}:0\le s\le t\}$ on a compact quantum group $\mathbb{G}$ are in one-to-one correspondence with those Markovian semigroups $\{S_t:0<t\}$ on the C$^*$-algebra $C(\mathbb{G})$ which are translation invariant in the sense that
\[
\Delta\circ S_t=(id\otimes S_t)\circ\Delta\qquad t>0.
\]
\end{thm}
To illustrate the main steps of the correspondence, notice first that the distributions of the process $\phi_t:=\Phi\circ j_{0t}$ form a continuous convolution semigroup on ${\rm Pol}(\mathbb{G})$
\[
\phi_0=\epsilon,\qquad \phi_s\ast\phi_t=\phi_{s+t},\qquad \lim_{t\to 0^+}\phi_t(a)=\epsilon (a)\qquad a\in {\rm Pol}(\mathbb{G})
\]
and that the {\it generating functional} of the process is then defined as
\[
G:D(G)\to\mathbb{C}\qquad G(a):=\frac{d}{dt}\phi_t(a)\Bigr|_{t=0}
\]
on a dense domain $D(G)\subseteq {\rm Pol}(\mathbb{G})$. From it one can reconstruct the distribution of the process as a convolution exponential
\[
\phi_t={\rm exp}_*(tG):=\epsilon +\sum_{n=1}^\infty \frac{t^n}{n!}G^{*n}\qquad t>0,
\]
a semigroup on ${\rm Pol}(\mathbb{G})$ by
\[
S_ta=\phi_t\ast a\qquad a\in {\rm Pol}(\mathbb{G}),\quad t>0
\]
and its formal generator $L:{\rm Pol}(\mathbb{G})\to {\rm Pol}(\mathbb{G})$ as $L(a):=G\ast a$. Then one checks that the semigroup extends to a strongly continuous, translation invariant Markovian semigropup on the C$^*$-algebra $C(\mathbb{G})$ and that its generator is the closure of $L$. Moreover, the distribution and the generating functional can be written directly in terms of the semigroup and its generator
\[
\phi_t=\epsilon\circ S_t,\qquad G=\epsilon\circ L.
\]
The KMS-symmetry of the semigroup of a L\'evy process can checked using the generating functional as follows
\begin{thm}
Let $\{S_t:t>0\}$ be the Markovian semigroup of a L\'evy process $\{j_{s,t}:0\le s\le t\}$ on compact quantum group $\mathbb{G}$. The following properties are then equivalent
\vskip0.1truecm\noindent
i) the semigroup is $(\sigma^h,-1)$-KMS symmetric\\
ii) the generating functional is invariant under the action of the unitary antipode
\[
G=G\circ R
\]
on the Hopf $^*$-algebra ${\rm Pol}(\mathbb{G})$.
\end{thm}
If the above conditions are verified then one can proceed to construct the Dirichlet form associated to the L\'evy process. The differential structure of these Dirichlet forms and the generating functional can be described in terms of the Sch\"urmann cocycle (see [CFK]) but we do not pursue it here.
\vskip0.2truecm\noindent
Rather, we prefer to conclude this section with examples of Dirichlet forms on a class of compact quantum groups whose spectrum has been completely determined with application to the approximation properties treated in a previous section.
\subsubsection{Free orthogonal Quantum Groups}
The universal C$^*$-algebra $C_u(O^+_N)$ of the free orthgonal quantum group of Wang $O^+_N$, $N\ge 2$,  is generated by a set of $N^2$ self-adjoint elements $\{v_{jk}:j,k=1,\cdots, N\}$ subject to the relations which ensure that the matrix $[v_{jk}]$ is unitary
\[
\sum_{l=1}^N v_{lj}v_{lk}=\delta_{jk}=\sum_{l=1}^N v_{jl}v_{kl}
\]and where a coproduct is defined as $\Delta v_{jk}:=\sum_{l=1}^N v_{lj}\otimes v_{lk}$. The Haar state is a trace which is faithful on the Hopf algebra but not on $C_u(O^+_N)$ so that the L\'evy semigroup is considered on the {\it reduced} C$^*$-algebra $C_r(O^+_N)$, defined by the GNS representation of the Haar state. The set of equivalence classes of irreducible, unitary co-representations is indexed by $\mathbb{N}$. Denoting by $\{U_s:s\in\mathbb{N}\}$ the Chebyshev polynomial on the interval $[-N,N]$ defined recursively as
\[
U_0(x)=1,\qquad U_1(x):=x,\qquad U_n(x)=xU_{n-1}(x)-U_{n-2}\qquad n\ge 2,
\]
a generating functional is then defined by
\[
G(u^n_{jk}):=\delta_{jk}\frac{U'_n(N)}{U_n(N)}\qquad \quad j,k=1,\cdots,U_n(N),\quad n\in\mathbb{N}.
\]
It can be proved that the associated Dirichlet form has {\it discrete spectrum} whose eigenvectors are the coefficients $u^n_{jk}$ of the irreducible, unitary co-representations and such that the corresponding eigenvalues and multiplicities are
\[
\lambda_n:=\frac{U'_n(N)}{U_n(N)},\qquad m_n:=(U_n(N))^2.
\]
By the results of a previous section, this implies that the von Neumann algebras $L^\infty(C^*_r(O^+_N),\tau)$ generated by the GNS representation of the Haar trace states, all have the Haagerup Property. In particular, however, since for $N=2$ one has $\lambda_n=\frac{n(n+2)}{6}$ and $m_n=(n+1)^2$, it results that $L^\infty(C^*_r(O^+_2),\tau)$ is amenable. The amenability of the free orthogonal quantum groups have been proved for the first time by M. Brannan [Bra].

\subsubsection{Property (T) of locally compact quantum groups and boundedness of Dirichlet forms}
We conclude this exposition describing succinctly a recent result of A. Skalski and A. Viselter [SV] connecting the Property (T) of the von Neumann algebra of a quantum group to the boundedness of translation invariant Dirichlet forms. Their framework is more general than the one treated in this section as they consider the {\it locally compact quantum groups} $\mathbb{G}=(M,\Delta, \varphi_H)$, in the von Neumann algebra setting, of J. Kustermans and S. Vaes [KV].\\
The main difference with respect to the S.L. Woronowicz theory of compact quantum groups is that the Haar weight  $\varphi_H$ (in general no more a state) is, together with the coproduct operation $\Delta$, part of the structure of a locally compact quantum group $\mathbb{G}$. This causes a lack of certain common, dense, natural domain for generators, generating functionals and quadratic forms so that a subtler analysis is required.\\
The von Neumann algebra $M$ (resp. its standard space $L^2(M)$) is often indicated as $L^\infty(\mathbb{G})$ (resp. $L^2(\mathbb{G})$) or as $L^\infty(\mathbb{G},\varphi_H)$ (resp. $L^2(\mathbb{G}, \varphi_H)$) to emphasize the reference to the chosen Haar weight.\\
From the point of view of potential theory, the unboundedness of the Plancherel weight necessitates of the extension of the theory of Dirichlet forms with respect to weights on von Neumann algebras, developed by S. Goldstein and J.M. Lindsay in [GL3] (and amended in [SV Appendix]). We do not describe the details of this theory here but we just notice that in case the Plancherel weight $\varphi_H$ is a trace we may use the theory illustrated in Section 4.
\vskip0.2truecm\noindent
The following result, obtained in [SV Theorem 4.6], characterizes the Property (T) of von Neumann algebras of separable locally compact quantum groups (defined in [F] for discrete quantum groups and for general locally compact ones in [DFSW]) in terms of a spectral property of the completely Dirichlet forms.
\begin{thm}
Let $\mathbb{G}$ be a locally compact quantum group such that $L^2(\mathbb{G},\varphi_H)$ is separable. Then the following properties are equivalent
\vskip0.2truecm\noindent
i) the von Neumann algebra $L^\infty(\mathbb{G},\varphi_H)$ has the property (T)
\vskip0.2truecm\noindent
ii) any translation invariant completely Dirichlet form on $L^2(\mathbb{G}, \varphi_H)$ is bounded.
\end{thm}
As in the compact case, the translation invariance of the Dirichlet form may be expressed as the invariance of the associated generating functional with respect to the unitary antipode.

\renewcommand{\abstractname}{Acknowledgements}
\begin{abstract}
The author wishes to thanks the referee for her/his careful and patient work.
\end{abstract}

%-----------------------------------------------------------------------------------------------------------------
%\end{enumerate}
%------------------------------------------------------------------------------------------------------------------------------------------
%\end{document}
% ----------------------------------------------------------------

\newpage
\normalsize
\section{References}

%\begin{center} \bf REFERENCES\end{center}

\end{document}